\documentclass{tmsce}

\usepackage{lipsum}
\usepackage{amsmath,amsfonts,amsthm}
\usepackage{algorithmic}
\usepackage{algorithm}
\usepackage{array}
\usepackage[caption=false,font=normalsize,labelfont=sf,textfont=sf]{subfig}
\usepackage{textcomp}
\usepackage{stfloats}
\usepackage{url}
\usepackage{verbatim}
\usepackage{graphicx}
\usepackage{cite}
\usepackage{xcolor}
\allowdisplaybreaks   % 允许所有 display 公式跨页
\usepackage{mwe} % provides example-image (remove later when you use your real images)
\newtheorem{proposition}{Proposition}[section]
\newtheorem{lemma}{Lemma}[section]
\newtheorem{remark}{Remark}[section]
\newtheorem{theorem}{Theorem}[section]
\newtheorem{definition}{Definition}[section]
\newtheorem{corollary}{Corollary}[section]
\numberwithin{equation}{section}

\title{Entropy-Based Characterization of  Fluctuations in Stochastically Perturbed Integrable Hamiltonian Systems}

\authors{Chen Wang$^{1}$ and Yong Li$^{1,2}$}
\affiliation{$^{1}$School of Mathematics, Jilin University, Changchun, 130012, People's Republic of China\\
	$^{2}$Center for Mathematics and Interdisciplinary Sciences, Northeast Normal University,
	Changchun, 130024, People's Republic of China\\
}
	

\begin{document}
\maketitle

\begin{abstract}
We study the entropy characteristics of stochastic fluctuations arising from randomly perturbed integrable Hamiltonian systems. Starting from  action-angle dynamics, we derive a discrete fluctuation process and investigate its asymptotic randomness from both dynamical and information-theoretic perspectives.

\begin{itemize}
	\item On the dynamical side, we establish an entropy variational principle for the induced shift dynamics on an energy-constrained infinite-dimensional path space. We show that, among all invariant probability measures satisfying the asymptotic energy constraint inherent to the system, the maximal entropy rate is attained by a Gaussian product measure. This establishes an entropy variational principle for stochastic dynamical trajectories under an asymptotic energy constraint.
	\item On the statistical side, we characterize the asymptotic Gaussian behavior of accumulated fluctuations through information-theoretic convergence. Beyond weak convergence given by the central limit theorem, we prove convergence of relative entropy and total variation distance between the  fluctuation distributions and their Gaussian limits. 
\end{itemize}
\end{abstract}

\printkeywords
%\printdates

% optional separator line (single-column)
\tmsceendfrontmatter

\section{Introduction}
\subsection{Stochastic Perturbations and  Introduction of Entropy}
Integrable Hamiltonian systems form a fundamental class of deterministic dynamical systems, where the presence of conserved quantities leads to highly organized motions and strong constraints on long-time evolution \cite{MR2036760}.
In many physical and mathematical models, however, such regular dynamics are inevitably affected by stochastic perturbations arising from external fluctuations and modeling uncertainties \cite{MR1201269,MR1723992}.
 Understanding how stochastic perturbations modify the long-time behavior of integrable systems and lead to emergent statistical structures is an important problem in stochastic dynamics.

To characterize the statistical and dynamical structures induced by such perturbations, existing studies have extensively explored the asymptotic behaviors of the systems, ranging from fluctuation limit theorems to Lyapunov-type characteristics of perturbed trajectories
 \cite{MR1894102,MR2062922,MR4801501,MR466791}. While limit theorems and ergodic invariants have provided profound insights into these aspects individually, they often address the distributional and dynamical features in isolation. Consequently, a systematic characterization that unifies these perspectives for fluctuation processes in perturbed Hamiltonian systems remains largely unexplored.

Entropy offers a natural and powerful framework to bridge these gaps, quantifying uncertainty and information generation in both probability theory and dynamical systems. In information theory, Shannon entropy and relative entropy provide rigorous metrics for distinguishing probability distributions \cite{MR1122806,MR4687403}, while in various dynamical systems, topological entropy quantifies the global orbit complexity, which is intrinsically linked to the Kolmogorov–Sinai entropy (the rate of information production along trajectories) via the variational principle \cite{MR4687027,MR4840048,MR2350155,MR3244290,MR1819189,MR648108}. However, applying these concepts to stochastic Hamiltonian systems presents a unique challenge: unlike standard stochastic processes defined a priori, the fluctuation processes here emerge dynamically from the underlying perturbed Hamiltonian flow. This work addresses this challenge by developing an entropy-based characterization that captures both the asymptotic randomness and the intrinsic dynamical structure of these emergent fluctuations.

\subsection{Main Results and Their Dynamical Interpretation}
	The central object of this study is the normalized fluctuation process
\begin{equation*}
	\hat{\theta}_{t}=\frac{\theta_{t}-\mathbb{E}\left[\theta_{t}\right]}{\sqrt{t}}, \ \ \ \ t\ge 0,
\end{equation*}
generated by stochastically perturbed integrable Hamiltonian
systems. More precisely, its discrete-time observations $\left\{\hat{\theta}_{n}\right\}_{n\in \mathbb{N}}$. In action-angle coordinates $ \left(I,\theta\right)\in \mathbb{R}^{m}\times\mathbb{T}^{d}$,  the dynamics are governed by
\begin{equation}
	dI_{t}=m_{I}d\xi_{t},\ \ \ \ d\theta_{t}=\omega\left(I_{t}\right)dt+m_{\theta}d\eta_{t},
\end{equation}
where $\xi_{t},\eta_{t}$ are independent stochastic processes
with intensities $m_I, m_\theta$. We assume the frequency $\omega$ and noise coefficients are sufficiently regular to ensure a finite asymptotic variance 
\begin{equation*}
		\tilde{\sigma}^{2}:=\lim_{n\to \infty}\frac{1}{n}\sum_{k=0}^{n-1}\mathrm{Var}\left(\hat{\theta}_{k}\right).
\end{equation*} 

To analyze the complexity of these fluctuations from a dynamical systems perspective, we employ a lattice-based coarse-graining. Rather than focusing on metric properties of the raw trajectories, we consider a symbolic encoding where the continuous fluctuation variables are mapped onto a discrete lattice $\Xi_{i}=\left\{a_{i}+jl_{i},j\in \mathbb{Z}\right\}$ via
		\begin{equation*}
		\tilde{\theta}_{k,i}:= a_{i}+l_{i} \cdot\mathrm{round}\left(\frac{\hat{\theta}_{k,i}-a_{i}}{l_{i}}\right),\ \ \ \ i=1,2,\cdots,d,
	\end{equation*}
	where $\mathrm{round}\left(\cdot\right)$ represents the nearest integer to $\cdot$, $a_{i},l_{i}>0$. 
	 This discretization allows us to treat the fluctuation process as a symbolic dynamical system. It provides a rigorous framework to quantify the orbit complexity, specifically, the rate at which the system generates new dynamical patterns—consistent with the concept of Kolmogorov-Sinai (KS) entropy.
	
Building on this framework, our main results establish a comprehensive picture of the emergent complexity in the system, rigorously connecting microscopic frequency fluctuations to macroscopic diffusion.
 We first determine the theoretical maximum of dynamical randomness under the physical energy constraint (Theorem 3.1). Let 
	\begin{equation*}
		\mathcal{V}^{(d)}_{\infty}\left(\tilde{\sigma}^{2}\right):=\left\{x=\left(x_{1},x_{2},\cdots\right)\in \left(\mathbb{R}^{d}\right)^{\mathbb{N}^{+}}:\limsup_{n\to \infty}\frac{1}{n}\sum_{k=1}^{n}\left\|x_{k}\right\|_{d}^{2}\le \tilde{\sigma}^{2} \right\}.
	\end{equation*}
We prove that the maximum KS entropy of system (1.1) is $\log\left(2\pi\mathrm{e}\tilde{\sigma}^{2}\right)/2$, and this upper bound is achieved by the law of an independent $d$-dimensional Gaussian noise. 
This establishes a sharp upper bound on the loss of regularity.
Remarkably, we show that the topological entropy of the system, subject to this energy constraint, exactly saturates this Gaussian bound. This indicates that the Gaussian measure represents the maximal entropy state allowed by the prescribed energy constraint.
	
Beyond Theorem 3.1,	for the accumulated fluctuations  $S_{n}:=\sum_{k=1}^{n}\tilde{\theta}_{k}$, we establish the convergence to the Gaussian limit in the strongest information-theoretic sense (Theorem 4.1), that is,
	\begin{align*}
		\lim_{n \to \infty}\left[H\left(S_{n}\right)+\sum_{j=1}^{d}\log l_{j}-\frac{1}{2}\log\left(\left(2\pi\mathrm{e}\right)^{d}\cdot\mathrm{det}\left(\Sigma_{n}^{2}\right)\right)\right]=0,
	\end{align*}
	and $\lim_{n \to \infty}D_{KL}\left(S_{n}\mid\mid\mathcal{N}\left(0,\Sigma_{n}^{2}\right)\right)=0$, where $H,D_{\mathrm{KL}}$ denote respectively Shannon entropy and Kullback–Leibler divergence, $\Sigma^{2}_{n}$ is the associated covariance matrix.
	Furthermore, this convergence holds in the sense of total variation, i.e., there exists a quantized Gaussian variable $Z$ such that $\lim_{n \to \infty}\left\|S_{n}-Z\right\|_{TV}=0$, where $\left\|\cdot\right\|_{\mathrm{TV}}$ denotes the total variation distance. 
	This establishes a sharp entropic bound for the phase diffusion, the convergence $D_{KL} \to 0$ implies that the fluctuation distribution becomes asymptotically indistinguishable from the Gaussian reference distribution,
	the information content of the phase becomes indistinguishable from that of a purely diffusive process. 
	
	\subsection{Technical Challenges and Innovations}
	Implementing the dual-level entropy framework presented in Section 1.2 requires overcoming fundamental theoretical barriers that lie beyond the scope of classical ergodic theory and probability. Specifically, we address two major technical hurdles:

	\begin{itemize}
		\item \textbf{Extending variational principles to non-compact, energy-constrained path spaces.} 
		While variational principles for topological entropy on non-compact spaces have been explored in the literature (e.g., \cite{MR370542,MR2109476,MR1971209}), these classical frameworks crucially rely on local compactness or finite-dimensional approximations that are fundamentally incompatible with our infinite-dimensional, energy-constrained path space 
		$\mathcal{V}^{(d)}_{\infty}\left(\tilde{\sigma}^{2}\right)$ subject to asymptotic energy constraints. This lack of compactness renders standard topological entropy theories inapplicable.
		
	For this, we develop a rigorous projective limit framework.
	By projecting the dynamics onto finite-length path spaces and carefully controlling the entropy defect, we successfully establish a variational principle for the KS entropy in this non-compact setting, explicitly identifying the Gaussian measure as the unique maximizer.
	
	\item  \textbf{Information-Theoretic Convergence for Non-stationary, Discretized Dynamics.} Standard entropic Central Limit Theorems (e.g., Barron's theorem) heavily rely on the assumptions of stationarity or continuous state spaces \cite{MR815975,MR4687403,MR2408387}. In our case, the fluctuation process is intrinsically non-stationary, and our coarse-graining maps it to a discrete lattice, which introduces non-trivial quantization effects that invalidate standard relative entropy estimates.
	
	For this, we establish a novel discrete-time entropic CLT tailored for non-stationary, lattice-valued observables. By introducing a lattice-adapted Bernoulli decomposition inspired by \cite{MR4687403} and verifying a generalized Lindeberg condition, we prove that the quantization errors vanish asymptotically, allowing us to rigorously control the information-theoretic metrics. 
	\end{itemize}
	
	\subsection{Organization of the Paper}
The remainder of this paper is organized as follows.
In Section 2, we provide the necessary preliminaries concerning entropy concepts and symbolic notations used throughout the paper.

In Section 3, we present the finite-dimensional approximation scheme for the constrained path space $\mathcal{V}_{\infty}^{(d)}(\tilde{\sigma}^{2})$, leading to the proof of the variational principle (Theorem 3.1) that identifies the Gaussian noise as the unique KS entropy maximizer.  In Section 4, we develop a lattice‑based framework for the fluctuation sequence $\left\{\tilde{\theta}_{n}\right\}_{n\in\mathbb{N}^{+}}$. We prove a strong entropic central limit theorem (Theorem 4.1) and, recover the corresponding convergence for the original continuous process (Corollary 4.1).  Finally, Section 5 concludes the paper and discusses possible extensions of the present framework.

\section{Preliminaries}
\begin{definition}[Topological entropy]
	Let $(X,d)$ be a metric space and
	$T:X\rightarrow X$ be a continuous map.
	For $n\geq1$, define the Bowen metric
	\[
	d_n(x,y)
	=
	\max_{0\leq k<n}
	d(T^kx,T^ky).
	\]
	For a subset $Z\subset X$ and $\epsilon>0$, let
	$
	s(Z,n,\epsilon)
	$
	denote the maximal cardinality of an
	$(n,\epsilon)$-separated subset of $Z$ with respect to the
	Bowen metric $d_n$.
	The topological entropy of $T$ on the subset $Z$ is defined by
	\[
	h_{\mathrm{top}}(T,Z)
	=
	\lim_{\epsilon\to0}
	\limsup_{n\to\infty}
	\frac1n
	\log s(Z,n,\epsilon).
	\]
\end{definition}

\begin{definition}
[Kolmogorov–Sinai entropy]
Let $(X,\mathcal{B},\mu,T)$ be a measure-preserving dynamical system, 
i.e., $(X, \mathcal{B},\mu)$ is a probability space and $T :X\to X$ is measurable with $\mu\left(T^{-1}A\right)=\mu\left(A\right)$ for all $A\in\mathcal{B}$.  For a finite measurable partition $\mathcal{P}$ of $X$, define 
\begin{equation*}
	H_{\mu}(\mathcal{P}) = -\sum_{P\in\mathcal{P}} \mu\left(P\right) \log \mu\left(P\right), \quad \text{with } \ 0 \log 0 = 0.
\end{equation*}
The Kolmogorov-Sinai entropy of $T$ with respect to $\mathcal{P}$ is
\begin{equation*}
	h_{\mu}\left(T,\mathcal{P}\right)=\lim_{n\to \infty}\frac{1}{n}H_{\mu}\left(\bigvee_{k=0}^{n-1}T^{-k}\mathcal{P}\right),
\end{equation*}
and the Kolmogorov–Sinai entropy of $\left(T,\mu\right)$ is
\begin{equation*}
	h_{\mu}\left(T\right)=\sup_{\mathcal{P}}h_{\mu}\left(T,\mathcal{P}\right),
\end{equation*}
where the supremum is taken over all finite measurable partitions $\mathcal{P}$ of $X$.
\end{definition}
	\begin{definition} [Shannon entropy and differential entropy]
	For a discrete random variable $X$ taking values in a countable set $\mathcal{X}$, its Shannon entropy is 
	\begin{equation*}
		H\left(X\right) = -\sum_{x\in \mathcal{X}} \mathbb{P}\left(X=x\right)\log\mathbb{P}\left(X=x\right),\quad \text{with}\ 0 \log 0 = 0.
	\end{equation*}
	For a continuous random vector $Y\in \mathbb{R}^{d} $ with density $p\left(y\right)$, its differential entropy is 
	\begin{equation*}
		h\left(Y\right) = -\int_{\mathbb{R}^{d}} p\left(x\right) \log p\left(x\right) dx.
	\end{equation*}
	Conditional entropies $H(X\mid \cdot)$ and $h(Y\mid \cdot)$ are defined analogously 
	via the corresponding conditional distributions.
\end{definition}
\begin{definition} [Relative entropy (Kullback-Leibler divergence)]
	Let $P$ and $Q$ be two probability measures on the same measurable space $\left(\Omega,\mathcal{F}\right)$. If $P$ is absolutely continuous with respect to $Q$ (denotes $P \ll Q$), the relative entropy is defined as
	\begin{equation*}
		D_{\mathrm{KL}}\left(P\mid\mid Q\right)=\int_{\Omega}\log\left(\frac{dP}{dQ}\right)dP, 
	\end{equation*}
	where $dP/dQ$ is the Radon–Nikodym derivative  and we adopt the convention $0 \log 0 = 0$. If $P$ is not absolutely continuous with respect to $Q$, we set $D_{\mathrm{KL}}(P\mid Q)=+\infty$. Meanwhile,
	If $P$ and $Q$ are discrete distributions with probabilities $p(y)$, $q(y)$ on a countable set $\mathcal{Y}$, then
	\begin{equation*}
		D_{\mathrm{KL}}\left(P\mid\mid Q\right)=
		\sum_{y\in\mathcal{Y}}p\left(y\right)\log\frac{p\left(y\right)}{q\left(y\right)}.
	\end{equation*}
\end{definition}

\begin{definition} [Total variation distance] For two discrete probability distributions with probability mass functions $p$ and $q$ on the same countable set $\mathcal{X}$, the total variation distance is defined as
	\begin{equation*}
		\left\|p-q\right\|_{\mathrm{TV}}=\frac{1}{2}\sum_{x\in\mathcal{X}}\left|p\left(x\right)-q\left(x\right)\right|.
	\end{equation*}
	Equivalently, for the corresponding random variables $X$ and $Y$, we write $	\left\|X-Y\right\|_{\mathrm{TV}}=\left\|p-q\right\|_{\mathrm{TV}}$.
\end{definition}

\begin{definition} [Pinsker's inequality] 
	Let $P$ and $Q$ be two probability measures on the same measurable space $\left(\Omega,\mathcal{F}\right)$.
	Then
	\begin{equation*}
		\left\|P-Q\right\|_{\mathrm{TV}}\le \sqrt{\frac{1}{2}D_{KL}\left(P\mid\mid Q\right)}.
	\end{equation*}
\end{definition}

In the rest of this paper, we stress that for any $d$-dimensional vector $Z=(Z_{1},Z_{2},\cdots,Z_{d})$, we adopt the standard convention of treating such row vector notation as representing a column vector, unless otherwise specified. $\left\|\cdot\right\|_{d}$ denotes the Euclidean norm of $d$-dimensional. We denote $\mathcal{N}\left(\mu,\sigma^{2}\right)$ as the Gaussian distribution with expectation $\mu$ and variance $\sigma^{2}$. Let $\sim$ signify identical distribution, $\mathrm{Var}\left(X\right)$ represents the variance of a real random variable $X$, the indicator function of a set $A$ is denoted by $\mathbb{I}_{A}$, $\mathbf{1}_{d}$ denotes the $d$-dimensional identity matrix. We write 
$f\left(n\right)\asymp g\left(n\right)$ as $n\to \infty$ if there exist positive constants $c_{1},c_{2}$ such that $c_{1}\left|g\left(n\right)\right|\le \left|f\left(n\right)\right|\le c_{2} \left|g\left(n\right)\right|$ for all sufficiently large $n$.
Additionally, $\mathrm{Bern}\left(\cdot\right)$ represents the Bernoulli distribution and  $\left|M\right|=\mathrm{det}\left(M\right)$
denotes the determinant of a square matrix $M$.

\section{A Constrained Variational Principle for Shift Entropy}
In this section, we study an entropy variational problem associated with the fluctuation sequence
 $\left\{\tilde{\theta}_{n}\right\}$ through a path-space shift system. 
 Our goal is to establish a constrained variational principle for the invariant measures of this shift system under the asymptotic energy constraint introduced below. The continuous state space $\mathbb{R}^{d}$ and the global energy constraint place this problem beyond the classical finite-alphabet shift setting. We formulate the variational problem in a measure-theoretic framework based on shift-invariant probability measures and Kolmogorov–Sinai entropy, with weak compactness providing the key analytical tool.

To construct the associated dynamical system, we embed the main sequence into a 
path space equipped with the left shift transformation. We first consider the general $d$-dimensional setting and later specialize to $d=1$ for notational convenience. Define
	\begin{equation*}
	\tilde{\Xi}_{d}:=\left(\mathbb{R}^{d}\right)^{\mathbb{N}^{+}}
\end{equation*}
as the one-sided full shift space equipped with the product topology. The left shift operator $T:\tilde{\Xi}_{d}\to \tilde{\Xi}_{d}$ is defined by $\left(Tw\right)_{k}=w_{k+1}$. A point $x \in\tilde{\Xi}_{d}$ represents a complete fluctuation path, and its orbit under the shift is $\left\{T^{n}x\right\}_{n\ge0}$. Denote by $\mathcal{M}(\tilde{\Xi}_{d})$ the space of Borel probability measures endowed with the weak‑* topology, and by $\mathcal{M}(\tilde{\Xi}_{d}, T)$ the subset of $T$‑invariant measures. These invariant measures are considered as admissible probability laws on the path space and are not assumed to coincide with the original law of the fluctuation sequence.
		
To incorporate the energy constraint into the shift framework, we restrict the shift space to sequences that satisfy the same asymptotic quadratic bound.  	Since removing finitely many initial coordinates does not change the asymptotic quadratic average, the set $\mathcal{V}_{\infty}^{(d)}(\tilde{\sigma}^{2})$ is $T$‑invariant, i.e., if $x\in \mathcal{V}_{\infty}\left(\tilde{\sigma}^{2}\right)$, then $Tx\in \mathcal{V}_{\infty}\left(\tilde{\sigma}^{2}\right)$. Hence $(\mathcal{V}_{\infty}^{(d)}(\tilde{\sigma}^{2}), T)$ itself is a topological dynamical system. Lemma 3.1 and Proposition 3.1 below will show that any natural shift‑invariant measure obtained from the process $\left\{\hat{\theta}_{n}\right\}_{n\in \mathbb{N}^{+}}$ 
is supported on $\mathcal{V}_{\infty}^{(d)}(\tilde{\sigma}^{2})$. Consequently, the entropy analysis of the process can be conducted entirely within this constrained subsystem.

	For the entropy computation we need finite‑length approximations of $\mathcal{V}_{\infty}^{(d)}(\tilde{\sigma}^{2})$. For each $n\in \mathbb{N}^{+}$, we define
\begin{equation*}
	\mathcal{V}^{(d)}_{n}:=\left\{x^{(n)}=\left(x_{1},x_{2},\cdots,x_{n}\right)\in \mathbb{R}^{n\times d}:\frac{1}{n}\sum_{k=1}^{n}\left\|x_{k}\right\|_{d}^{2}\le E_{n}\right\},
\end{equation*} 
where $E_{n}:= n^{\beta_{n}}\left(\tilde{\sigma}^{2}+2\vartheta_{n}\right)$, $\left\{\vartheta_{n}\right\}_{n\in \mathbb{N}^{+}}$ is a fixed positive sequence tending to zero,  $\beta_{n}:=1/\left(n\log n\right)$ when $n\ge 2$ and $\beta_{n}=1$ when $n=1$.
Hence, for $n\ge 2$, the factor $n^{\beta_{n}}$ decreases monotonically to 1 as $n\to \infty$. Meanwhile, because $E_{n}\to \tilde{\sigma}^{2}$ (independently of the auxiliary sequences $\beta_{n}, \vartheta_{n}$), the set $\mathcal{V}_{n}$
consists of length‑$n$ blocks whose empirical second moment is at most $\tilde{\sigma}^{2}+o\left(1\right)$. the slight relaxation $E_{n}$ is introduced to obtain suitable compactness properties.
We equip $\mathcal{V}_{n}^{(d)}$ with an $\ell^{2}$ metric,
\begin{equation*}
	d_{n}^{(d)}\left(x,y\right) :=
	\sqrt{\frac{1}{E_{n}^{2}}\sum_{k=1}^{n}\left\|x_{k}-y_{k}\right\|_{d}^{2}},\ \ \ \ x,y\in\mathcal{V}^{(d)}_{n}.
\end{equation*}

The finite sets $\mathcal{V}_{n}^{(d)}$ will play the role of ``admissible words'' of length $n$ when we compute the topological entropy of the constrained system, analogous to counting words in a subshift of finite type. For the sake of clarity, we first carry out the entropy analysis in the one‑dimensional case ($d=1$), where notation is lighter. The extension to $d>1$ under the assumption of independent components is straightforward and will be indicated afterward. From now on we drop the superscript $(d)$ when $d=1$, writing simply $\mathcal{V}_{n},\mathcal{V}_{\infty},d_{n}$.

\subsection{Topological Entropy of the Energy-Constrained Shift}
This subsection presents
a variational principle for dynamics under an asymptotic energy constraint. The key is to approximate the infinite-dimensional constraint set $\mathcal{V}_{\infty}^{(d)}\left(\tilde{\sigma}^{2}\right)$ by a sequence of finite-dimensional sets $\left\{\mathcal{V}_{n}\right\}_{n\in \mathbb{N}^{+}}$, reducing the problem to a limit of classical variational principles.
We start with $d=1$. 

The classical topological entropy as in Definition 2.1 assumes a compact phase space. Our set $\mathcal{V}^{(d)}_{\infty}\left(\tilde{\sigma}^{2}\right)$ is non‑compact, and the natural metric on finite segments is the normalized $\ell ^{2}$ distance $d_{n}$ defined above. Therefore we adapt the definition as follows.

For $\epsilon>0$, let $\mathrm{sep}(n,\epsilon,\mathcal{V}_{n})$ be a maximal $\epsilon$-separated set in the metric space $\left(\mathcal{V}_{n},d_{n}\right)$,
its cardinality is denoted by $\left|\mathrm{sep}(n,\epsilon,\mathcal{V}_{n})\right|$. We introduce the following constrained topological entropy adapted to  the shift system $\left(\mathcal{V}_{\infty}\left(\tilde{\sigma}^{2}\right),T\right)$,
\begin{equation*}
	h_{\mathrm{top}}^{c}\left(T\right):=\lim_{\epsilon\to 0}\limsup_{n\to \infty}\frac{\log\left|\mathrm{sep}(n,\epsilon,\mathcal{V}_{n})\right|}{n}.
\end{equation*}

\begin{remark}
	Our definition of $h_{\mathrm{top}}^{c}(T)$ extends the classical notion of topological entropy, counting exponentially many distinguishable trajectories to the energy‑constrained and non compact space $\mathcal{V}_{\infty}^{(d)}\left(\tilde{\sigma}^{2}\right)$. Because the classical construction requires a compact phase space, we work instead with the finite-dimensional approximations $\mathcal{V}_{n}^{(d)}$, equipped with  
	metric $d_{n}$ that reflects the natural scaling of fluctuations. The resulting entropy rate quantifies the constrained capacity of the system, i.e., the exponential growth rate of admissible finite segments ($\mathcal{V}_{n}^{(d)}$ playing the role of ``words'').
	This adapted definition retains the key variational property of classical topological entropy: as stated in Theorem 3.1, $h_{\mathrm{top}}(\mathcal{V}^{(d)}_{\infty}\left(\tilde{\sigma}^{2}\right),T)$ equals the supremum of KS entropies over all shift-invariant measures supported on $\mathcal{V}_{\infty}^{(d)}\left(\tilde{\sigma}^{2}\right)$. Thus, while tailored to a non‑compact setting, our definition still provides an upper bound for the measure‑theoretic complexity of every admissible invariant measure.
\end{remark}

\begin{lemma}
	For any $n\ge 1$ and $\epsilon>0$, 
	\begin{equation*}
		\mathcal{V}_{n}\subset \bigcup_{x\in \mathrm{sep}\left(n,\epsilon,\mathcal{V}_{n}\right)}B_{n}\left(x,\epsilon\right),
	\end{equation*}
	where $	B_{n}\left(x,\epsilon\right):=\left\{y\in \mathcal{V}_{n}:d_{n}(x,y) \le \epsilon\right\}$.
\end{lemma}
\begin{proof}
	If some $z\in \mathcal{V}_{n}$ is not contained in any $B_{n}\left(x,\epsilon\right)$ with $x\in\mathrm{sep}\left(n,\epsilon,\mathcal{V}_{n}\right) $, 
	then $d_{n}\left(z,x\right)>\epsilon$ for all $x$, contradicting the maximality of that $\epsilon$-separated set. 
\end{proof}

\begin{proposition}
	Let $\mathrm{Vol}(\cdot)$ denote the $n$-dimensional Lebesgue measure. For the set $\mathcal{V}_{n}\subset\mathbb{R}^{n}$,
	\begin{equation*}
		\lim_{\epsilon\to 0}\limsup_{n\to \infty}\frac{1}{n}\left|\log\operatorname{Vol}\left(\mathcal{V}_{n}\right)-\log\left|\mathrm{sep}(n,\epsilon,\mathcal{V}_{n})\right|\right|=0.
	\end{equation*} 
\end{proposition}
\begin{proof}
	According to Lemma 3.1,  	the $\epsilon$-balls centered at points of $\mathrm{sep}\left(n,\epsilon,\mathcal{V}_{n}\right)$ cover  $\mathcal{V}_{n}$, i.e., 
	\begin{align*}
		\mathrm{Vol}\left(\mathcal{V}_{n}\right)\le \sum_{x\in\mathrm{sep}\left(n,\epsilon,\mathcal{V}_{n}\right) }	\mathrm{Vol}\left(B_{n}\left(x,\epsilon\right)\right)
		\le \mathrm{sep}(n,\epsilon,\mathcal{V}_{n})\cdot\max_{x\in \mathrm{sep}(n,\epsilon,\mathcal{V}_{n})}\mathrm{Vol}\left(B_{n}\left(x,\epsilon\right)\right).
	\end{align*}
	Taking any $x\in \mathrm{sep}\left(n,\epsilon,\mathcal{V}_{n}\right)$ and $y\in B_{n}\left(x,\epsilon\right)$, we have $d^{2}_{n}\left(x,y\right)=\sum_{k=1}^{n}\left(x_{k}-y_{k}\right)^{2}\le \epsilon^{2}E^{2}_{n}$.
	By Cauchy–Schwarz inequality, $	\left|x_{k}-y_{k}\right|\le \sqrt{\sum_{k=1}^{n}\left(x_{k}-y_{k}\right)^{2}}\le \epsilon E_{n}=n^{\beta_{n}/2}\epsilon\sqrt{\tilde{\sigma}^{2}+2\vartheta_{n}}$. 
	Thus, the volume of $B_{n}\left(x,\epsilon\right)$ can be bounded by that of a cube, i.e.,  $\mathrm{Vol}\left(B_{n}\left(x,\epsilon\right)\right)\le \left(2n^{\beta_{n}/2}\epsilon\sqrt{\tilde{\sigma}^{2}+2\vartheta_{n}}+1\right)^{n}$. 
	Thus, for sufficiently small $\epsilon>0$, under the fact that $n^{\beta_{n}}\to 1$ and $\vartheta_{n}\to 0$ as $n\to \infty$,
	we obtain
	\begin{align*}
		\frac{1}{n}\log\mathrm{Vol}\left(B_{n}\left(x,\epsilon\right)\right)&=\log\left(2n^{\beta_{n}/2}\epsilon\sqrt{\tilde{\sigma}^{2}+2\vartheta_{n}}+1\right)\\
		&=2n^{\beta_{n}/2}\epsilon\sqrt{\tilde{\sigma}^{2}+2\vartheta_{n}}+o\left(2n^{\beta_{n}/2}\epsilon\sqrt{\tilde{\sigma}^{2}+2\vartheta_{n}}+1\right).	
	\end{align*}
	This indicates that as $n\to \infty$,
	\begin{align*}
		&\frac{1}{n}\left(\log\mathrm{Vol}\left(\mathcal{V}_{n}\right)-\log\left|\mathrm{sep}\left(n,\epsilon,\mathcal{V}_{n}\right)\right|\right)\\
		\le& \frac{1}{n}\left(\log\left(\sum_{x\in \mathrm{sep}\left(n,\epsilon,\mathcal{V}_{n}\right)}\mathrm{Vol}\left(B_{n}\left(x,\epsilon\right)\right)\right)-\log\left|\mathrm{sep}\left(n,\epsilon,\mathcal{V}_{n}\right)\right|\right)\nonumber\\
		\le & \frac{1}{n}\log\max_{x\in
			\mathcal{V}_{n}}\mathrm{Vol}\left(B_{n}\left(x,\epsilon\right)\right)\nonumber\\
		\le&\log\left(2n^{(\beta_{n}+1)/2}\epsilon\sqrt{\tilde{\sigma}^{2}+2\delta_{n}}+o\left(n^{(\beta_{n}+1)/2}\epsilon\sqrt{\tilde{\sigma}^{2}+2\delta_{n}}\right)\right)\nonumber\\
		\to&2\epsilon\tilde{\sigma}+o\left(\epsilon\right).
	\end{align*}
	Therefore, we have
	\begin{equation}
		\lim_{\epsilon\to0}\lim_{n\to \infty}\frac{1}{n}\left(\log\left|\mathcal{V}_{n}\right|-\log\left|\mathrm{sep}\left(n,\epsilon,\mathcal{V}_{n}\right)\right|\right)\le 
		\lim_{\epsilon\to0}\left(2\epsilon\tilde{\sigma}+o\left(\epsilon\right)\right)=0.
	\end{equation}
	Similarly, considering  the $\epsilon$-balls $B_{n}(x,\epsilon/2)$ are disjoint for distinct $x\in\mathrm{sep}(n,\epsilon,\mathcal{V}_{n})$ (because $d_{n}(x,y)>\epsilon$). Hence
	\begin{equation*}
		\operatorname{Vol}\left(\mathcal{V}_{n}\right)\ge \sum_{x\in\mathrm{sep}(n,\epsilon,\mathcal{V}_{n})}\operatorname{Vol}\bigl(B_{n}(x,\frac{\epsilon}{2})\bigr)
		\ge \left|\mathrm{sep}(n,\epsilon,\mathcal{V}_{n})\right|\cdot\inf_{x\in\mathcal{V}_{n}}\operatorname{Vol}\bigl(B_{n}(x,\frac{\epsilon}{2})\bigr),
	\end{equation*}
	where $\mathrm{Vol}\left(B_{n}\left(x,\epsilon/2\right)\right)$ is always larger than that of a smaller cube $\left(n^{\beta_{n}/2}\epsilon\sqrt{\tilde{\sigma}^{2}+2\vartheta_{n}}/2\right)^{n}$.
	Taking logarithms, dividing by $n$, and letting $n\to\infty$ then $\epsilon\to0$, we derive              
	\begin{equation}
		\lim_{\epsilon\to0}\limsup_{n\to\infty}\frac{1}{n}\log\left|\mathrm{sep}(n,\epsilon,\mathcal{V}_{n})\right|
		\le \lim_{n\to\infty}\frac{1}{n}\log\operatorname{Vol}\left(\mathcal{V}_{n}\right).
	\end{equation}
	Combining (3.1)-(3.2), we finish our proof.
\end{proof}

\begin{proposition}
	For $d\ge 1$, the topological entropy $	h_{\mathrm{top}}^{c}\left(T\right)=d\log\left(2\pi\mathrm{e}\tilde{\sigma}^{2}\right)/2$.
\end{proposition}
\begin{proof}
By Stirling's formula, we obtain
	\begin{align*}
		\lim_{n\to \infty}\frac{1}{n}\log\mathrm{Vol}\left(\mathcal{V}_{n}\right)
		&= \lim_{n\to \infty}\frac{1}{n}\log\left(\frac{\pi^{n/2}\left(\sqrt{n^{1+\beta_{n}}\left(\tilde{\sigma}^{2}+2\vartheta_{n}\right)}\right)^{n}}{\sqrt{\pi n}\left(\frac{n}{2\mathrm{e}}\right)^{n/2}}\right)\\
		&=\lim_{n\to \infty}\frac{1}{2}\log\left(2\pi\mathrm{e}\left(\tilde{\sigma}^{2}+2\vartheta_{n}\right)\right)+\frac{1}{2}\lim_{n\to \infty}\beta_{n}\log n\\
		&=\frac{1}{2}\log\left(2\pi\mathrm{e}\tilde{\sigma}^{2}\right).
	\end{align*}
	Next, combining it with Proposition 3.1, we obtain that
	\begin{equation*}
		h_{\mathrm{top}}^{c}\left(T\right)=\lim_{\epsilon\to 0}\limsup_{n\to \infty}\frac{\log\left|\mathrm{sep}(n,\epsilon,\mathcal{V}_{n})\right|}{n}
		=\lim_{\epsilon\to 0}\limsup_{n\to \infty}\frac{\log\mathrm{Vol}\left(\mathcal{V}_{n}\right)}{n}=\frac{1}{2}\log\left(2\pi\mathrm{e}\tilde{\sigma}^{2}\right).	
	\end{equation*}
	Finally, for the $d$-dimensional case with independent components,
	the corresponding topological entropy is $d$ times the one-dimensional result.
	This completes the proof.
\end{proof}

	\subsection{Measure-theoretic Entropy of the Energy-Constrained Shift}
In classical symbolic dynamics, entropy is typically computed using finite partitions of a finite alphabet. For our continuous alphabet $\mathbb{R}^{d}$, we instead consider partitions induced by cubic grids $\mathcal{Q}_{\epsilon}=\left\{\text{cubes of side length}\ \epsilon\right\}$  and pass to the limit $\epsilon\to 0$. Although this approach is analytic rather than combinatorial, it preserves the dynamical idea of counting distinguishable trajectories.

In our setting, $\left(X,T\right)=\left( \tilde{\Xi}_{d} ,T \right)$ where $T$ is the left shift. Let $Z_{k} : \tilde{\Xi}_{d}\to \mathbb{R}^{d}$ be the $k$-th coordinate projection, $Z_{k}(x) = x_{k}$. For $\epsilon>0$, let 	$\mathcal{P}_{\epsilon} := \left\{Z_{0}^{-1}\left(A\right):A\in \mathcal{Q}_{\epsilon}\right\}$  be the corresponding partition of $\tilde{\Xi}$. Then $T^{-k}\mathcal{P}_{\epsilon}$ is the partition according to the $k$-th coordinate. As $\epsilon\to 0$, the family $\left\{\mathcal{P}_{\epsilon}\right\}$ generates the Borel $\sigma$‑algebra under $T$, hence by the KS theorem, 	
\begin{equation*}
	h_{\mu}\left(T\right)=\lim_{\epsilon\to 0}	h_{\mu}\left(T,\mathcal{P}_{\epsilon}\right)=\lim_{\epsilon\to 0}\lim_{N\to \infty}\frac{1}{N}H_{\mu}\left(\bigvee_{k=0}^{N-1}T^{-k}\mathcal{P}_{\epsilon}\right).
\end{equation*}
We are particularly interested in 
$T$-invariant probability measures supported on $\mathcal{V}^{(d)}_{\infty}\left(\tilde{\sigma}^{2}\right)$. For such a measure $\mu$,  the measure-theoretic entropy $h_{\mu}\left(T\right)$ quantifies the average unpredictability per iteration along $\mu$-typical orbits.

\begin{proposition}
	For any shift-invariant measure $\mu$ supported on $\mathcal{M}\left(\mathcal{V}^{(d)}_{\infty}\left(\tilde{\sigma}^{2}\right),T\right)$,
	\begin{equation*}
		h_{\mu}\left(T\right)\le\frac{d}{2} \log\left(2\pi\mathrm{e}\tilde{\sigma}^{2}\right)	.
	\end{equation*}
\end{proposition}
\begin{proof}
	We give the argument for $d=1$, the extension to $d>1$ follows by applying the one-dimensional bound to each coordinate and using the sub-additivity of entropy. 	Given any measure $\mu\in\mathcal{M}\left(\mathcal{V}_{\infty}\left(\tilde{\sigma}^{2}\right),T\right)$, we denote by $Y_{k}:\tilde{\Xi}\to\mathbb{R}$ the $k$-th coordinate projection with $Y_{k}\left(y\right)=y_{k}$. Since $\mu$ is $T$-invariant, the process $\left\{Y_{n}\right\}_{n\in\mathbb{N}^{+}}$ is stationary.   Moreover, by the definition of $\mathcal{V}_{\infty}\left(\tilde{\sigma}^{2}\right)$,
	\begin{equation*}
		\limsup_{n\to\infty}\frac{1}{n}\sum_{k=1}^{n}Y_{k}^{2}\left(y\right)\le \tilde{\sigma}^{2},\ \ \text{for}\ \mu-\text{a.e.}\ y. 
	\end{equation*}
	Applying Birkhoff's ergodic theorem to the observable $y\mapsto y_{1}^{2}$, we obtain a $T$-invariant limit 
	\begin{equation*}
		\sigma^{2}\left(y\right):=\lim_{n\to \infty}\frac{1}{n}\sum_{k=1}^{n}Y_{k}^{2}\left(y\right)\le \tilde{\sigma}^{2},\ \ \text{for}\ \mu-\text{a.e.}\ y.
	\end{equation*}
	Integrating gives $\mathbb{E}_{\mu}\left[Y_{1}^{2}\right]=\int_{\tilde{\Xi}}\sigma^{2}\left(y\right)d\mu\left(y\right)\le \tilde{\sigma}^{2}$.  Thus every invariant measure on $\mathcal{V}_{\infty}\left(\tilde{\sigma}^{2}\right)$ has a second moment bounded by $\tilde{\sigma}^{2}$. 
	
	Let $\mu_{1}$ be the one-dimensional marginal distribution of $Y_{1}$
	under $\mu$. If $\mu_{1}$ is absolutely continuous with respect to Lebesgue measure, denote its density by $p$ and its differential entropy is  $	h_{\mu}\left(Y_{1}\right):=-\int_{\mathbb{R}} p\left(y\right) \log p\left(y\right) dy$.
	If $\mu_{1}$ is singular, we set $h_{\mu}\left(X_{1}\right)=-\infty$. A basic information-theoretic fact 	states that for any random variable $Z$ with $\mathbb{E}\left[Z^{2}\right]\le\sigma^{2} $, $h\left(Z\right)\le \log\left(2\pi\mathrm{e}\sigma^{2}\right)/2$, with equality iff $Z$ is Gaussian with variance $\sigma^{2}$.
	Applying this to $Z=Y_{1}$ yields $h_{\mu}\left(Y_{1}\right)\le \log\left(2\pi\mathrm{e}\mathbb{E}_{\mu}\left[Y_{1}^{2}\right]\right)/2\le \log\left(2\pi\mathrm{e}\tilde{\sigma}^{2}\right)/2$. 
	If $\mu_{1}$ is singular, the inequality holds trivially because the left‑hand side equals $-\infty$.
	
	For each $n\ge1$, write $Y_{1}^{n}=\left(Y_{1},Y_{2},\cdots,Y_{n}\right)$.  Stationarity and the chain rule for differential entropy give $h_{\mu}\left(Y_{1}^{n}\right)\le \sum_{k=1}^{n}h_{\mu}\left(Y_{k}\mid Y_{1}^{k-1}\right)\le nh_{\mu}\left(Y_{k}\mid Y_{1}^{k-1}\right)\le nh_{\mu}\left(Y_{k}\right)$, where $h_{\mu}\left(Y_{k}\mid Y_{1}^{k-1}\right)$ denotes conditional differential entropy. 	Therefore, for a stationary process $\left\{Y_{k}\right\}_{k=1}^{n}$ with finite marginal entropy, we have $h_{\mu}\left(T\right)=\lim_{n\to\infty}h_{\mu}\left(Y_{1}^{n}\right)/n\le h_{\mu}\left(Y_{1}\right) \le \log\left(2\pi\mathrm{e}\tilde{\sigma}^{2}\right)/2$.
	The identity remains valid (by a standard truncation argument) even when $h_{\mu}\left(Y_{1}\right)=-\infty$, in that case both sides equal $-\infty$.
	
	The case $d>1$ is obtained by applying the one-dimensional result to each component $Y_{1}^{(i)}$ of $Y_{1} = (Y_{1}^{(1)},\cdots,Y_{1}^{(d)})$. 	We omit this part for simplicity and finish our proof now.
\end{proof}

We now construct a shift-invariant measure in the one-dimensional case that attains the entropy upper bound $\log(2\pi \mathrm{e}\tilde{\sigma}^{2})/2$. The extension to $d > 1$ follows by taking the product measure, yielding the entropy rate $d\log(2\pi \mathrm{e}\tilde{\sigma}^{2})/2$ which matches the upper bound.

For each n, let $K_{n}$ be uniformly distributed on $\left\{1,2,\cdots,n\right\}$ and independent of the finite segment  $(\hat{\theta}_{1},\hat{\theta}_{2},\cdots,\hat{\theta}_{n})$. Define the periodic extension \begin{equation*}
	V^{(n)} = (V_{1}^{(n)}, V_{2}^{(n)}, \cdots, V_{n}^{(n)}, V_{1}^{(n)},V_{2}^{(n)}, \cdots),
\end{equation*}
where $V_{t}^{(n)}:= \hat{\theta}_{\left(K_{n}+t-1\right)\ \mathrm{mod}\ n+1}$.
Let $\mu_{n}$ be the distribution of $V^{(n)}$ on $\tilde{\Xi}$, i.e., for any Borel set $A \subset \tilde{\Xi}$, $\mu_{n}\left(A\right)=\mathbb{P}\left(V^{(n)}\in A\right)$. 
Because each $V^{(n)}$ is $n$-periodic and fulfilled with uniformly bounded second moment in $n$, we can choose a sequence of positive numbers $\left\{s_{n}\right\}_{n\in\mathbb{N}^{+}}$ such that $\sum_{k=1}^{\infty}s_{k}<\infty$. For any $\epsilon > 0$, we define
\begin{equation*}
	K_{\epsilon}:=\left\{x\in \tilde{\Xi}: \sum_{k=1}^{\infty}s_{k}x_{k}^{2}\le J_{\epsilon}\right\},
\end{equation*}
where $J_{\epsilon}>0$  will be chosen momentarily. 
By Markov's inequality, there exists some $C_{1}>0$ such that 
\begin{equation*}
	\mu_{n}\left(\sum_{k=1}^{\infty}s_{k}x_{k}^{2}>J_{\epsilon}\right)\le \frac{1}{J_{\epsilon}}\sum_{k=1}^{\infty}s_{k}\mathbb{E}_{\mu_{n}}\left[x_{k}^{2}\right]\le \frac{C_{1}\sum_{k=1}^{\infty}s_{k}}{J_{\epsilon}}.
\end{equation*}
This leads to the choice of $J_{\epsilon}$ such that $C_{1} \sum_{k=1}^{\infty} s_{k}/J_{\epsilon}< \epsilon$. Then for all $n$,  $\mu_{n}\left(K_{\epsilon}\right)\ge 1-\epsilon$. 
Meanwhile, each $K_{\epsilon}$ is compact because it is a closed subset of the compact product $\prod_{k=1}^{\infty} [-\sqrt{J_{\epsilon}/s_{k}}, \sqrt{J_{\epsilon}/s_{k}}]$ from Tychonoff theorem.
Hence, $\left\{\mu_{n}\right\}_{n\in\mathbb{N}^{+}}$ is tight on the Polish space $\tilde{\Xi}$, and by Prokhorov's theorem, it contains a weakly convergent subsequence.
Passing to this subsequence (and relabelling $n$), we obtain a weak limit $\mu_{\infty}:=\lim_{n\to \infty}\mu_{n}$.

	\begin{lemma}
	The measure $\mu_{\infty}$ is supported on the energy-constrained set,
	$\mu_{\infty}\left(\mathcal{V}_{\infty}\left(\tilde{\sigma}^{2}\right)\right)=1$.
\end{lemma}
\begin{proof}
	For each fixed $n$, Birkhoff's ergodic theorem applied to $\mu_{n}$ gives that 
	\begin{equation*}
		\lim_{m\to \infty}R_{m}\left(x\right):=\lim_{m\to \infty}\frac{1}{m}\sum_{k=1}^{m}x_{k}^{2}=E_{\mu_{n}}\left[x_{1}^{2}\right]
		= \frac{1}{n}\sum_{k=1}^{n}\mathbb{E}\left[\left|\hat{\theta}_{k}\right|^{2}\right]\le \tilde{\sigma}^{2}+o\left(1\right),\ \ \mu_{n}-\text{a.e.}
	\end{equation*}	
	The function $R_{m}(x)$ is continuous on $\tilde{\Xi}$ (it depends only on finitely many coordinates). Hence, the upper‑limit function $\mathcal{R}(x):= \limsup_{m\to\infty} R_{m}(x)$ 
	satisfies $\mu_{n}(\left\{x:R\left(x\right)\le \tilde{\sigma}^{2}+\epsilon\right\}) = 1$ for all sufficiently large $n$ (depending on any $\epsilon>0$).
	
	The set $B_{\epsilon} := \left\{x:\mathcal{R}(x)\le \tilde{\sigma}^{2}+\epsilon\right\}$ is closed (because $\mathcal{R}$ is upper‑semicontinuous).
	Tightness of $\left\{\mu_{n}\right\}_{n\in\mathbb{N}^{+}}$ provides a compact set $K_{\delta} \subset \tilde{\Xi}$ with  $\mu_{n}(K_{\delta}) \ge 1-\delta$ for all $n\in \mathbb{N}^{+}$. The intersection $B_{\epsilon} \bigcap K_{\delta}$ satisfies that
	\begin{equation*}
		B_{\epsilon}\bigcap K_{\delta}=\left\{x\in K_{\delta}:\mathcal{R}\left(x\right)\le\tilde{\sigma}^{2}+\epsilon\right\}=K_{\delta}\setminus\left\{x\in K_{\delta}:\mathcal{R}\left(x\right)>\tilde{\sigma}^{2}+\epsilon\right\}
	\end{equation*}
	is compact (closed subset of a compact set). By the Portmanteau theorem, \begin{equation*}
		\mu_{\infty}\left(B_{\epsilon}\bigcap K_{\delta}\right)\ge\limsup_{n\to \infty}\mu_{n}\left(B_{\epsilon}\bigcap K_{\delta}\right)\ge1-\delta.  \end{equation*}
	Letting $\delta \to 0$ gives $\mu_{\infty}\left(B_{\epsilon}\right)=1$ for every $\epsilon>0$. Finally, 
	we take $\epsilon_{l}=1/l$ and intersect over $l\in \mathbb{N}^{+}$ and yield 
	\begin{equation*}
		\bigcap_{l=1}^{\infty}B_{1/l}=\bigcap_{l=1}^{\infty}\left\{x:\limsup_{n\to \infty} R_{n}\left(x\right)\le \tilde{\sigma}^{2}+\frac{1}{l}\right\}=\left\{x:\limsup_{n\to \infty} R_{n}\left(x\right)\le \tilde{\sigma}^{2}\right\}=\mathcal{V}_{\infty}\left(\tilde{\sigma}^{2}\right).
	\end{equation*}
	Hence, $\mu_{\infty}\left(\bigcap_{l=1}^{\infty}B_{1/l}\right)=\mu_{\infty}\left(\mathcal{V}_{\infty}\left(\tilde{\sigma}^{2}\right)\right)=1$.
	We finish our proof now. 
\end{proof}
The following results attains the maximal entropy rate among all shift‑invariant probability measures supported on $\mathcal{V}_\infty^{(d)}(\tilde{\sigma}^{2})$.	

\begin{proposition}
	Let $\mu^{(d)}_{\infty}:=\mu_{\infty}^{\otimes d}$.
	Then $\mu_{\infty}^{(d)}$ belongs to $ \mathcal{M}\left(\mathcal{V}^{(d)}_{\infty}\left(\tilde{\sigma}^{2}\right),T\right)$, and its entropy rate is $	h_{\mu^{(d)}_{\infty}}\left(T\right)=d\log\left(2\pi\mathrm{e}\tilde{\sigma}^{2}\right)/2$. 
\end{proposition}
\begin{proof}
	Consider $d=1$ first. For any $M>0$, define the truncation $f_{M}(z)=\operatorname{sgn}(z)\min\left\{|z|,M\right\}$ at level $M>0$, which is continuous and bounded. Under the weak convergence of $\mu_{n}$ to $\mu_{\infty}$, we have $\int f_{M}d\mu_{n}-\int f_{M}d\mu_{\infty}\to 0$ as $n\to \infty$.
	Because $|x-f_{M}(x)|\le |x|\cdot\mathbb{I}_{{|x|>M}}$,  applying this deduction to $\mu_{n}$
	yields that 
	\begin{align*}
		\int	\left|f_{M}\left(x\right)-x\right|d\mu_{n}\left(x\right)	
		\le \sqrt{\int x^{2}d\nu_{k}\left(x\right)}\cdot \sqrt{\mathbb{P}\left(\left|\hat{\theta}_{k}\right|>M\right)}
		\le \frac{\mathbb{E}_{\mu_{n}}\left[x_{1}^{2}\right]}{M}.
	\end{align*}
	The same estimation follows if we substitute $\mu_{\infty}$ with $\mu_{n}$. Therefore, for any $\epsilon>0$, given that $M$ and then $N$ are both large enough, for all $n>N$, 
	\begin{align*}
		\left|\int xd\mu_{\infty}\left(x\right)-\int xd\mu_{n}\left(x\right)\right|
		&	\le \left|\int xd\mu_{\infty}\left(x\right)-\int f_{M}\left(x\right)d\mu_{\infty}\left(x\right)\right|\nonumber\\
		&\qquad\qquad+\left|\int f_{M}\left(x\right)d\mu_{\infty}\left(x\right)-\int f_{M}\left(x\right)d\mu_{n}\left(x\right)\right|\nonumber\\
		&\qquad\qquad\qquad\qquad+\left|\int f_{M}\left(x\right)d\mu_{n}\left(x\right)-\int xd\bar{\mu}_{n}\left(x\right)\right|\\
		&<\epsilon.
	\end{align*}
	In other words, we have
	\begin{equation*}
		\int xd\mu_{\infty}\left(x\right)=\lim_{n\to \infty}\int xd\mu_{n}\left(x\right)=\lim_{n\to \infty}\frac{1}{n}\sum_{k=1}^{n}\mathbb{E}\left[\hat{\theta}_{k}\right],
	\end{equation*}
	which is zero by the centering assumption.
	The variance is handled analogously. Define
	\begin{equation*}
		g_{M}\left(x\right)=\left\{
		\begin{aligned}
			& \left(x-\int yd\mu_{\infty}\left(y\right)\right)^{2},\ \ \mathrm{for}\ \left|x-\int yd\mu_{\infty}\left(y\right)\right|\le M,  \\ 
			&M, \ \ \text{othersise}.
		\end{aligned}
		\right. 
	\end{equation*}
	Repeating the same three-term splitting argument with $g_{M}$ in place of $x$, we obtain for any $\epsilon>0$, there exist $M_{0}=M_{0}(\epsilon)$ and $N_{0}=N_{0}(\epsilon)$ such that whenever $M>M_{0}$ and $n>N_{0}$,
	\begin{equation*}
		\left|\int\left(x-\int yd\mu_{\infty}\left(y\right)\right)^{2}d\mu_{\infty}\left(x\right)-\int\left(x-\int yd\mu_{\infty}\left(x\right)\right)^{2}d\mu_{n}\left(x\right)\right|
		<\epsilon.	
	\end{equation*}
	Hence, we have
	\begin{align*}
		\int \left(x-\int yd\mu_{\infty}\left(y\right)\right)^{2}d\mu_{\infty}\left(x\right)
		=\lim_{n\to \infty}\frac{1}{n}\sum_{k=1}^{n}\mathbb{E}\left[\left(\hat{\theta}_{k}-\lim_{m\to \infty}\frac{1}{m}\sum_{j=1}^{m}\mathbb{E}\left[\hat{\theta}_{j}\right]\right)^{2}\right],
	\end{align*}
	which simplifies to
	\begin{equation*}
		\int x^{2}d\mu_{\infty}\left(x\right)=\lim_{n\to \infty}\frac{1}{n}\sum_{k=1}^{n}\mathbb{E}\left[\hat{\theta}_{k}^{2}\right]=\tilde{\sigma}^{2}.
	\end{equation*}
	
	Fix an integer $m\ge 1$ and indices $t_{1},\dots,t_{m}$. 
	For $n$ large enough, under $\mu_{n}$, the random vector $(x_{t_{1}},\dots,x_{t_{m}})$ 
	is obtained by selecting $m$ distinct coordinates from $(\hat{\theta}_{1},\cdots,\hat{\theta}_{n})$ uniformly at random 
	(due to the random starting point $K_{n}$). 
	Since $\left\{\hat{\theta}_{n}\right\}_{n\in \mathbb{N}^{+}}$ 
	satisfies
	$\sum_{k=1}^n\text{Var}\left(\hat{\theta}_{k}\right)/n\to\tilde{\sigma}^{2}$, 
	a multivariate central limit theorem implies that
	$	\mathrm{Law}_{\mu_{n}}(x_{t_{1}},\cdots,x_{t_{m}}) \to\mathcal{N}(0, \tilde{\sigma}^{2} \mathbf{1}_{m})$ weakly
	as $n\to\infty$. Consequently, the finite $m$-dimensional marginal of $\mu_\infty$ is $\mathcal{N}(0,\tilde{\sigma}^2 \mathbf{1}_{m})$.
	By the Kolmogorov extension theorem, there is a unique measure on $\tilde{\Xi}$ with these 
	finite-dimensional marginals, namely the product measure $\mu^{\ast}:=\bigotimes_{k\in\mathbb{N}^{+}} \mathcal{N}(0,\tilde{\sigma}^2)$, that is, $\mu^{\ast}=\mu_{\infty}$ are both stationary Gaussian measures with the same mean zero and variance $\tilde{\sigma}^{2}$. For such a process the entropy is expressed as 
	\begin{equation*}
		h_{\mu_{\infty}}\left(T\right)=h_{\mu^{\ast}}\left(T\right)=\frac{1}{2}\log\left(2\pi\mathrm{e}\tilde{\sigma}^{2}\right).
	\end{equation*}
	For any finite indices $J_{m}:=\left\{t_{1},t_{2},\cdots,t_{m}\right\}\subset\mathbb{N}$,  
	define the project $\pi_{J_{m}}:\tilde{\Xi}\to\mathbb{R}^{m}$ by $\pi_{J_{m}}(y)=(y_{t_{1}},\dots,y_{t_{m}})$, and let $T$ be the left shift $(Ty)_{k}=y_{k+1}$, $\mu_{J_{m}}=\mathcal{N}\left(0,\tilde{\sigma}^{2}\mathbf{1}_{m}\right)$ is the corresponding finite-dimensional marginal. For any cylinder set $A=\pi_{J_{m}}^{-1}(B)$ with $B\subset\mathbb{R}^{m}$, we obtain that
	\begin{equation*}
		T^{-1}A=\left\{y:Ty\in A\right\}=\left\{y:\left(y_{t_{1}+1},y_{t_{2}+1},\cdots,y_{t_{m}+1}\right)\in B\right\}=\pi^{-1}_{J_{m}+1}\left(B\right),
	\end{equation*}
	where $J_{m}+1=\left\{t_{1}+1,t_{2}+1,\cdots,t_{m}+1\right\}$. Therefore,
	\begin{eqnarray*}
		\mu^{\ast}\left(T^{-1}A\right)=\mu^{\ast}\left(\pi^{-1}_{J_{m+1}}\left(B\right)\right)=
		\mu_{J_{m+1}}\left(B\right).\end{eqnarray*}
	Because the projection of $\mu_{J_{m}+1}$ onto the first $m$ dimensions is the same as that of $\mu_{J_{m}+1}\left(B\right)=\mu_{J_{m}}\left(B\right)=\mu^{\ast}\left(A\right)$. Thus $\mu^{\ast}$ is $T$-invariant (and non‑degenerate because its marginals are Gaussian with positive variance). 
	Consequently, the weak limit $\mu_{\infty}$ obtained earlier coincides with $\mu^{\ast}$, belonging to $\mathcal{M}\left(\mathcal{V}_{\infty}\left(\tilde{\sigma}^{2}\right),T\right)$, and therefore
	\begin{equation*}
		h_{\mu_{\infty}}\left(T\right)=h_{\mu^{\ast}}\left(T\right)=\frac{1}{2}\log\left(2\pi\mathrm{e}\tilde{\sigma}^{2}\right).
	\end{equation*}
	For the $d$-dimensional product measure $\mu_{\infty}^{(d)}=\mu_{\infty}^{\otimes d}$, independence of the coordinates gives that
	\begin{equation*}
		h_{\mu_{\infty}^{(d)}}\left(T\right)=\sum_{j=1}^{d}h_{\mu_{\infty}}\left(T\right)=\frac{d}{2}\log\left(2\pi\mathrm{e}\tilde{\sigma}^{2}\right),
	\end{equation*} 
	which coincides with the upper bound established in Proposition 3.3. We finish our proof.
\end{proof}

Now we establish the trajectory-level entropy characterization associated with the constrained fluctuation dynamics.
	\begin{theorem}
	For $d\ge1$, the equality
	\begin{equation*}
		\sup_{\mu\in \mathcal{M}\left(\mathcal{V}^{(d)}_{\infty}\left(\tilde{\sigma}^{2}\right),T\right)}h_{\mu}\left(T\right)=h_{\mathrm{top}}^{c}=\frac{d\log\left(2\pi\mathrm{e}\tilde{\sigma}^{2}\right)}{2}
	\end{equation*}
	holds, 
	and the supremum is attained by the product measure $\mu^{(d)}_{\infty}$.
\end{theorem}
\begin{remark}
	The quantity
	$
h_{\mathrm{top}}^{c}
	(T)
	$
	measures the exponential complexity of admissible fluctuation paths
	under the prescribed asymptotic energy constraint. Through this constrained
	variational principle, this complexity is equivalently characterized
	by the maximal information production rate among invariant probability
	measures supported on the constrained path space.
\end{remark}

\section{Strong Entropy Convergence of Accumulated Fluctuations}
This section establishes the information-theoretic foundation for the Gaussian limit. We prove that, under the non-stationary dynamics, the entropy of the macroscopic observable $S_{n}$ from model (1.1) converges strongly to its continuous (Gaussian) limit, extending classical entropy convergence to our setting and bridging the discrete analysis with the continuum. Note that the value of each $\tilde{\theta}_{n}$ exactly depends on parameter $l$. In the context of dealing with it, we temporarily assume $l$ to be fixed, 
along with related quantities, which will be treated without explicit $l$-dependence in our notation. However, when we return to the original continuous system, the dependence on $l$ will be reinstated, and we will emphasize it through appropriate superscripts ``$\left(l\right)$''.

Recall that for a discrete random variable $Y$ defined on $\Xi_{i}$ for any $i=1,2,\cdots,d$ with probability mass function $p\left(y\right)=\mathbb{P}\left(Y=y\right)$ (with expectation $\mu$ and finite variance $\sigma^{2}$), its relative entropy to a quantized Gaussian $q$ is $D_{KL}\left( p\mid \mid q\right)	=\sum_{y}p\left(y\right)\log\left( p\left(y\right)/q\left(y\right)\right)$, 
where $q$ is derived from the Gaussian density $\phi$ of $\mathcal{N}\left(\mu,\sigma^{2}\right)$ via 
\begin{equation*}
	q\left(a+kl\right)=\int_{a+kl}^{a+\left(k+1\right)l}\phi\left(y\right)dy,\ \ \ \ k\in \mathbb{Z}.
\end{equation*}	
Thus, it can be immediately calculated that $	\log\phi\left(x\right)=-\left(x^{2}+\log\left(2\pi\right)\right)/2$. When the second argument of $D_{KL}$
is the quantized Gaussian distribution associated with $Y$
as defined above, we shall simply write $D_{KL}\left(Y\right)$ instead of 
$D_{KL}\left( p\mid \mid q\right)$. The same abbreviated notation will be used for other random variables appearing later.

\begin{lemma}The partial sum $S_{n}$ and the standardised sum $S_{n}':= S_{n}/s_{n}$ satisfy that, as $n\to \infty$,
	\begin{equation*}
		D_{KL}\left(S_{n}'\right)=\frac{1}{2}\log\left(2\pi\mathrm{e}\right)-\left[H\left(S_{n}\right)-\log\frac{s_{n}}{l}\right] +O\left(\frac{1}{s_{n}}\right),
	\end{equation*}
	where $s_{n}:=\sqrt{\sum_{k=1}^{n}\mathrm{Var}\left(\tilde{\theta}_{k}\right)}$ and for any $n\ge 1$, the $O\left(1/s_{n}\right)$ can be absolutely bounded by $l/s_{n}+l^{2}/2s_{n}^{2}$.
	Moreover, let $U$ be an independent uniform random variable on $(-1/2,1/2)$. Then as $n\to \infty$,
	\begin{equation*}
		D_{KL}\left(S_{n}'\right)=D_{KL}\left(S_{n}'+\frac{l}{s_{n}}U\right)+O\left(\frac{1}{s_{n}}\right),
	\end{equation*}
	where the error term $O\left(1/s_{n}\right)$ is absolutely bounded by $l/s_{n}+13l^{2}/\left(24s_{n}^{2}\right)$.
\end{lemma}
\begin{proof}
	First, the sums $S_{n}$ and $S_{n}'$ take values in $\left\{na+kl:k \in \mathbb{Z}\right\}$ and $\left\{(na+kl)/s_{n}:k \in \mathbb{Z}\right\}$
	respectively. By denoting $p,p_{S_{n}'}$ and $q$ as the probability mass functions of $S_{n}$, $S_{n}'$ and the quantised Gaussian in the definition of $D_{KL}\left(S_{n}'\right)$, we obtain that for any $k\in \mathbb{Z}$ and some $\nu_{k}\in \left[k,k+1\right]$,
	\begin{equation}
		q\left(\frac{na+kl}{s_{n}}\right)=\int_{\frac{na+kl}{s_{n}}}^{\frac{na+\left(k+1\right)l}{s_{n}}}\phi\left(x\right)dx=\frac{l}{s_{n}}\cdot\phi\left(\frac{na+\nu_{k}l}{s_{n}}\right).\nonumber
	\end{equation}
	Therefore, on the one hand, 
	\begin{align}
		D_{KL}\left(S_{n}'\right)&=\sum_{k\in \mathbb{Z}}p\left(na+kl\right)\cdot\log\left[\frac{p\left(na+kl\right)}{\frac{l}{s_{n}}\phi\left(\frac{na+\nu_{k}l}{s_{n}}\right)}\right],
	\end{align}
	while on the other hand, we have
	\begin{align}
		&\frac{1}{2}\log\left(2\pi \mathrm{e}\right)+\log\left(\frac{s_{n}}{l}\right)\nonumber\\	=&\frac{1}{2s_{n}^{2}}\sum_{k\in\mathbb{Z}}p_{S_{n}'}\left(\frac{na+kl}{s_{n}}\right)\cdot\left(na+kl\right)^{2}+\frac{\log\left(2\pi\right)}{2}\cdot\sum_{k\in\mathbb{Z}}p\left(na+kl\right)-\log\left(\frac{l}{s_{n}}\right)\nonumber\\
		=&\frac{1}{2s_{n}^{2}}\sum_{k\in\mathbb{Z}}p\left(na+kl\right)\cdot\left(na+kl\right)^{2}+\frac{\log\left(2\pi\right)}{2}\cdot\sum_{k\in\mathbb{Z}}p\left(na+kl\right)-\log\left(\frac{l}{s_{n}}\right)\nonumber\\
		=&-\log\left(\frac{l}{s_{n}}\right)+\frac{1}{2}\sum_{k\in \mathbb{Z}}p\left(na+kl\right)\left[\frac{\left(na+kl\right)^{2}}{s_{n}^{2}}+\log\left(2\pi\right)\right]\nonumber\\
		=&-\sum_{k\in \mathbb{Z}}p\left(na+kl\right)\cdot\log\left[\frac{l}{s_{n}}\cdot\phi\left(\frac{na+kl}{s_{n}}\right)\right].
	\end{align}
	Combining (4.1)-(4.2), we obtain that
	\begin{align}
		&\left|D_{KL}\left(S_{n}'\right)+H\left(S_{n}\right)-  \frac{1}{2}\log\left(2\pi \mathrm{e}\right)-\log\left(\frac{s_{n}}{l}\right)\right|\nonumber\\
		=&\left|\sum_{k\in \mathbb{Z}}p\left(na+kl\right)\cdot\log\left[\frac{p\left(na+kl\right)}{\frac{l}{s_{n}}\cdot\phi\left(\frac{na+\nu_{k}l}{s_{n}}\right)}\right]+\sum_{k\in \mathbb{Z}}p\left(na+kl\right)\cdot\log\left[\frac{\frac{l}{s_{n}}\cdot\phi\left(\frac{na+kl}{s_{n}}\right)}{p\left(na+kl\right)}\right]\right|\nonumber\\
		=&\left|\sum_{k\in \mathbb{Z}}p\left(na+kl\right)\cdot\log\left(\frac{\phi\left(\frac{na+kl}{s_{n}}\right)}{\phi\left(\frac{na+\nu_{k}l}{s_{n}}\right)}\right)\right|\nonumber\\
		\le &\frac{1}{2s_{n}^{2}}\sum_{k\in \mathbb{Z}}p\left(na+kl\right)\cdot\left|\left(na+kl\right)^{2}-\left(na+\nu_{k}l\right)^{2}\right|\nonumber\\
		\le& \frac{l}{s_{n}^{2}}\cdot\left[\sum_{k\in \mathbb{Z}}p\left(na+kl\right)\right]^{1/2}\cdot\left[\sum_{k\in \mathbb{Z}}p\left(na+kl\right)\cdot\left|na+kl\right|^{2}\right]^{1/2}+\frac{l^{2}}{2s_{n}^{2}}\nonumber\\
		=&\frac{l}{s_{n}^{2}}\cdot\left(\mathbb{E}\left(S_{n}\right)^{2}\right)^{1/2}+\frac{l^{2}}{2s_{n}^{2}}\nonumber\\
		=&\frac{l}{s_{n}}+\frac{l^{2}}{2s_{n}^{2}},
	\end{align}
	which leads to 
	\begin{equation}
		H\left(S_{n}\right)-\log\left(\frac{s_{n}}{l}\right)\le \frac{1}{2}\log\left(2\pi\mathrm{e}\right)+\frac{l}{s_{n}}\left(1+\frac{l}{2s_{n}}\right),
	\end{equation}
	as $D_{KL}\left(S_{n}\right)\ge 0$. More precisely, with the definition of $U$ above, we have
	\begin{eqnarray*}
		H\left(S_{n}\right)-\log\left(\frac{s_{n}}{l}\right)&=	H\left(S_{n}'\right)-\log\left(\frac{s_{n}}{l}\right)\le \frac{1}{2}\log\left[2\pi\mathrm{e}\left(1+\frac{l^{2}}{12s_{n}}\right)\right].
	\end{eqnarray*}
	Therefore, it can be estimated from (4.3)-(4.4) that
	\begin{align*}
		&\left|D_{KL}\left(S_{n}'\right)-D_{KL}\left(S_{n}'+\frac{l}{s_{n}}U\right)\right|\\
		\le & \left|D_{KL}\left(S_{n}'+\frac{l}{s_{n}}U\right)-\frac{1}{2}\log\left(2\pi\mathrm{e}\right)+H\left(S_{n}\right)-\log\left(\frac{s_{n}}{l}\right)\right|+\frac{l}{s_{n}}+\frac{l^{2}}{2s_{n}^{2}}\\
		= & \left|\frac{1}{2}\log\left(2\pi\mathrm{e}\left(1+\frac{l^{2}}{12s_{n}^{2}}\right)\right)-l\left(S_{n}'+\frac{l}{s_{n}}U\right)\right.\\
		&\qquad\qquad\qquad\left.-\frac{1}{2}\log\left(2\pi\mathrm{e}\right)+\left[H\left(S_{n}\right)-\log\left(\frac{s_{n}}{l}\right)\right]\right|+\frac{l}{s_{n}}+\frac{l^{2}}{2s_{n}^{2}}\\
		=&\left|\frac{1}{2}\log\left(1+\frac{l^{2}}{12s_{n}^{2}}\right)\right|+\frac{l}{s_{n}}+\frac{l^{2}}{2s_{n}^{2}}\\
		\le& \frac{l}{s_{n}}+\frac{13l^{2}}{24s_{n}^{2}}.	
	\end{align*}
	Our proof is complete now.
\end{proof}

\begin{remark}
	It is important to note that for discrete random variables, convergence of the relative entropy to 0 is not equivalent to convergence of the Shannon entropy. 
	Lemma 4.1 explicitly establishes the relationship between these two entropies of sequence $\left\{\tilde{\theta}_{n}\right\}_{n\in\mathbb{N}^{+}}$, which crucially depends on the step size $l$. 
\end{remark}

To deal with the lattice values of $\left\{\tilde{\theta}_{n}\right\}_{n\in\mathbb{N}^{+}}$, we start with the case when it is a sequence of independent random variables with the same Bernoulli distribution $\mathrm{Bern}\left(1/2\right)$ 
taking values in (a subset of) $\Xi$ with $l=1$, as it represents a simplest symmetric discrete distribution. We thus immediately arrive at the following conclusion.
\begin{proposition}\cite{MR4687403}
	For all $n\ge 2$,
	\begin{equation*}
		\left|\left[H\left(S_{n}\right)-\log\sqrt{n}\right]-\frac{1}{2}\log\left(\frac{1}{2}\pi\mathrm{e}\right)\right|\le \frac{4}{\sqrt{n}},
	\end{equation*}
	while $D_{KL}\left(S_{n}'\right)\le 8/\sqrt{n}$.
\end{proposition}

In fact, if we relax the condition on the identical distribution of $\left\{\tilde{\theta}_{n}\right\}_{n\in\mathbb{N}^{+}}$, that is, it is a sequence of independent random variables with the $i$-th variable distributed as $\mathrm{Bern}\left(p_{i}\right)$ for some $p_{i}\in [0,1]$, we therefore obtain the following result.

\begin{proposition}
	Suppose that $\left\{p_{i}\right\}_{i=1}^{n}$ satisfies that $\lim_{n \to \infty}\sum_{i=1}^{n}p_{i}\left(1-p_{i}\right)=\infty$,
	while for some $\beta\in (0,1/2)$,
	\begin{equation}
		\lim_{n\to \infty}\frac{
			\log n\cdot\sum_{i=1}^{n}p_{i}}{\left(\sum_{i=1}^{n}p_{i}\left(1-p_{i}\right)\right)^{2\beta+1}}=0.\end{equation}
	Then for all $n\ge 2$, $\left|H\left(S_{n}\right)-\log\left(2\pi\mathrm{e}\left(\sum_{i=1}^{n}p_{i}(1-p_{i})\right)\right)/2\right|$ is bounded by
	\begin{equation*}
		\min\left\{\frac{1}{24\sqrt{\sum_{i=1}^{n}p_{i}\left(1-p_{i}\right)}},2\mathrm{exp}\left\{-r\left(n\right)+\log n+\log r \left(n\right)\right\}\right\},	
	\end{equation*}
	where $r(n)$ is defined in (4.10).
	Furthermore, $D_{KL}\left(S_{n}'\right)$ has a upper bound expressed as
	\begin{align*}
		&\frac{1}{\sqrt{\sum_{i=1}^{n}p_{i}\left(1-p_{i}\right)}}+\frac{1}{2\sum_{i=1}^{n}p_{i}\left(1-p_{i}\right)}\\
		&\qquad\qquad
		+\min\left\{\frac{1}{24\sqrt{\sum_{i=1}^{n}p_{i}\left(1-p_{i}\right)}},2\mathrm{exp}\left\{-r\left(n\right)+\log n+\log r \left(n\right)\right\}\right\}.
	\end{align*}
\end{proposition} 
\begin{proof}
	According to Lemma 4.1, we have
	\begin{equation}
		H\left(S_{n}\right)-\log\left(s_{n}\right)\le \frac{1}{2}\log\left[2\pi\mathrm{e}\left(1+\frac{1}{12s_{n}}\right)\right]
		\le \frac{1}{2}\log\left(2\pi\mathrm{e}\right)+\frac{1}{24s_{n}},
	\end{equation}
	where at this point, it can be immediately calculated that  $s_{n}^{2}=\sum_{i=1}^{n}p_{i}\left(1-p_{i}\right)$. 
	Due to the loss of identical distribution, for the estimation of the lower bound of $H\left(S_{n}\right)$, we are unable to follow the same idea as in \cite{MR4687403}. However, since for $i=1,2,\cdots,n$, we have $   \left|\tilde{\theta}_{i}-p_{i}\right|\le \max\left(p_{i},1-p_{i}\right)\le 1$,
	then for any $\delta>0$, there exists some $N\ge 1$ satisfying that when $n\ge N$, $\delta\left(\sum_{i=1}^{n}p_{i}(1-p_{i})\right)^{1/2}>1$, so that for $n\ge N$, $\sum_{i=1}^{n}\mathbb{E}\left[\left(\tilde{\theta}_{i}-p_{i}\right)^{2}\mid\mathbb{I}_{\left|\tilde{\theta}_{i}-p_{i}\right|>\delta s_{n}}\right]/s_{n}^{2}=0$,
	that is, \begin{equation*}
		\lim_{n \to \infty}\frac{1}{s_{n}^{2}}\cdot\sum_{i=1}^{n}\mathbb{E}\left[\left(\tilde{\theta}_{i}-p_{i}\right)^{2}\mid\mathbb{I}_{\left|\tilde{\theta}_{i}-p_{i}\right|>\delta s_{n}}\right]=0.
	\end{equation*}
	Thus, we may introduce some function $c(n)$ satisfying $\lim_{n \to \infty}c\left(n\right)=\infty$, so that for any $k$ belonging to interval
	\begin{equation}
		\mathbb{Z}\cap\left[\sum_{i=1}^{n}p_{i}-c\left(n\right)s_{n},\sum_{i=1}^{n}p_{i}+c\left(n\right)s_{n}\right],
	\end{equation}  
	it can be estimated that
	\begin{equation}
		\mathbb{P}\left(S_{n}=k\right)=\frac{1}{\sqrt{2\pi s_{n}^{2}}}\mathrm{exp}\left\{-\left(k-\sum_{i=1}^{n}p_{i}\right)^{2}/2s_{n}^{2}\right\}
		\left(1+o\left(1\right)\right).
	\end{equation}
	Here, aiming at the uniformly bounded sequence $\left\{\tilde{\theta}_{n}\right\}_{n\in\mathbb{N}^{+}}$, the scaling function $c(n)$ can be specifically picked as $s_{n}^{2\alpha}$ for $\alpha\in (0,1/2)$ in the following.   It will exactly determine the width of the central region (4.7).
	Meanwhile, when $k\in \mathbb{Z}\cap\left(\left[\sum_{i=1}^{n}p_{i}-s_{n}^{2\alpha+1},\sum_{i=1}^{n}p_{i}+s_{n}^{2\alpha+1}\right]\right)^{\mathrm{c}}  $, i.e., $k$ satisfies that $\left|k-\sum_{i=1}^{n}p_{i}\right|>s_{n}^{2\alpha+1}$, 
	then for some $u\in \mathbb{R}$,
	\begin{align}
		\mathbb{P}\left(S_{n}=k\right)&=2\mathbb{P}\left(\sum_{i=1}^{n}\tilde{\theta}_{i}>\sum_{i=1}^{n}p_{i}+s_{n}^{2\alpha+1}\right)
		\le\frac{2\mathbb{E}\left(\mathrm{exp}\left\{u\sum_{i=1}^{n}\tilde{\theta}_{i}\right\}\right)}{\mathrm{exp}\left\{u\left(\sum_{i=1}^{n}p_{i}+s_{n}^{2\alpha+1}\right)\right\}}
	:=2\mathrm{e}^{g\left(u\right)},
	\end{align}
	where $g\left(u\right):=-u\left(\sum_{i=1}^{n}p_{i}+s_{n}^{2\alpha+1}\right)+\sum_{i=1}^{n}p_{i}\left(\mathrm{e}^{u}-1\right)$.
	Thus, it can be calculated that the solution of $g'\left(u\right)=0$
	corresponds to $u^{\ast}=\ln\left(\left(1+s_{n}^{2\alpha+1}\right)/\sum_{i=1}^{n}p_{i}\right)$,
	and, $u^{\ast}$ satisfies that $g''\left(u^{\ast}\right)>0$. 
	Since $\alpha<1/2$, we obtain 
	\begin{equation*}
		\frac{s_{n}^{2\alpha+1}}{\sum_{i=1}^{n}p_{i}}\le\left(\sum_{i=1}^{n}p_{i}\right)^{\alpha-\frac{1}{2}}\to0,\  \ \mathrm{as}\ n \to \infty,
	\end{equation*}
	so that $u^{\ast}=s_{n}^{2\alpha+1}/\sum_{i=1}^{n}p_{i}+O\left(s_{n}^{4\alpha+2}/\left(\sum_{i=1}^{n}p_{i}\right)^{2}\right)$,
	which leads to
	\begin{align*}
		g\left(u^{\ast}\right)&=-\ln\left(1+\frac{s_{n}^{2\alpha+1}}{\sum_{i=1}^{n}p_{i}}\right)\cdot\left(\sum_{i=1}^{n}p_{i}+s_{n}^{2\alpha+1}\right)+\sum_{i=1}^{n}p_{i}\left(\mathrm{e}^{u^{\ast}}-1\right)\\
		&=s_{n}^{2\alpha+1}-\left(\frac{s_{n}^{2\alpha+1}}{\sum_{i=1}^{n}p_{i}}-\frac{s_{n}^{4\alpha+2}}{\left(\sum_{i=1}^{n}p_{i}\right)^{2}}\right)\cdot\left(\sum_{i=1}^{n}p_{i}+s_{n}^{2\alpha+1}\right)\\
		&=\frac{s_{n}^{6\alpha+3}}{2\left(\sum_{i=1}^{n}p_{i}\right)^{2}}-\frac{s_{n}^{4\alpha+2}}{2\sum_{i=1}^{n}p_{i}}\\
		&=\frac{s_{n}^{4\alpha+2}}{2\sum_{i=1}^{n}p_{i}}\cdot\frac{s_{n}^{2\alpha+1}-\sum_{i=1}^{n}p_{i}}{\sum_{i=1}^{n}p_{i}}\\
		&=O\left(-\frac{s_{n}^{4\alpha+2}}{2\sum_{i=1}^{n}p_{i}}\right).
	\end{align*}
	Back to (4.9), for $k$ satisfying $\left|k-\sum_{i=1}^{n}p_{i}\right|>s_{n}^{2\alpha+1}$,
	we obtain that
	\begin{equation*}
		\mathbb{P}\left(S_{n}=k\right)\le 2\cdot\inf_{u}\mathrm{exp}\left\{g\left(u\right)\right\} = 2\mathrm{exp}\left\{-\frac{s_{n}^{4\alpha+2}}{2\sum_{i=1}^{n}p^{2}_{i}+2s_{n}^{2}}\right\}.	
	\end{equation*}
	If $\sum_{i=1}^{n}p_{i}^{2}=o\left(s_{n}^{2}\right)$, then it turns out that
	\begin{equation*}
		\mathbb{P}\left(S_{n}= k\right)= 
		2\mathrm{exp}\left\{-s_{n}^{4\alpha}\right\}\to 0,\ \ \mathrm{as}\ n \to \infty.
	\end{equation*} 
	If $\sum_{i=1}^{n}p_{i}^{2}$ has the same order with $s_{n}^{2}$, then it can be similarly deduced that
	\begin{equation*}
		\mathbb{P}\left(S_{n}= k\right)=O\left(\mathrm{exp}\left\{-s_{n}^{4\alpha}\right\}\right)\to 0,\ \ \mathrm{as}\ n \to \infty.
	\end{equation*}
	In particular, 
	if we define
	\begin{equation}
		r\left(n\right):=2\mathrm{exp}\left\{s_{n}^{2\alpha+1}/2\sum_{i=1}^{n}p_{i}\right\},
	\end{equation}
	then from (4.5), it can be further estimated as $\lim_{n \to \infty}r\left(n\right)/\log n=\infty$, and, as $n \to \infty$,
	\begin{align*}
		\sum_{|k-\sum_{i=1}^{n}p_{i}|>s_{n}^{2\alpha+1}} \mathrm{e}^{-r\left(n\right)}\cdot\left(r\left(n\right)-\log2\right)&\le 2n\mathrm{e}^{-r\left(n\right)}\cdot\left(r\left(n\right)-\log 2\right)\\
		&	\le 2\mathrm{exp}\left\{-r\left(n\right)+\log n+\log r\left(n\right)\right\}\\
		&		\to0,
	\end{align*}
	which will not affect our subsequent estimations exactly. Therefore, it can be calculated from (4.8) that
	\begin{align}
		H\left(S_{n}\right)&=-\sum_{|k-\sum_{i=1}^{n}p_{i}|\le s_{n}^{2\alpha+1}}P\left(S_{n}=k\right)\cdot\log\mathbb{P}\left(S_{n}=k\right)\nonumber\\
		&\qquad\qquad-\sum_{|k-\sum_{i=1}^{n}p_{i}|>s_{n}^{2\alpha+1}}P\left(S_{n}=k\right)\cdot\log\mathbb{P}\left(S_{n}=k\right)\nonumber\\
		&\le \frac{1}{2}\log\left(2\pi\mathrm{e}s_{n}^{2}\right)+o\left(1\right)+2\mathrm{exp}\left\{-l\left(n\right)+\log n+\log \left(n\right)\right\}.	
	\end{align}
	Meanwhile, as long as $s_{n}^{2\alpha+1}$ (as well as $n$) is large enough, we are led to
	\begin{align}
		H\left(S_{n}\right)&=-\sum_{|k-\sum_{i=1}^{n}p_{i}|\le s_{n}^{2\alpha+1}}P\left(S_{n}=k\right)\cdot\log\mathbb{P}\left(S_{n}=k\right)\nonumber\\
		&\qquad\qquad-\sum_{|k-\sum_{i=1}^{n}p_{i}|>s_{n}^{2\alpha+1}}P\left(S_{n}=k\right)\cdot\log\mathbb{P}\left(S_{n}=k\right)\nonumber\\
		&\ge -\sum_{|k-\sum_{i=1}^{n}p_{i}|\le s_{n}^{2\alpha+1}}\frac{1}{\sqrt{2\pi s_{n}^{2}}}\cdot\mathrm{exp}\left\{-\left(k-\sum_{i=1}^{n}p_{i}\right)^{2}/2s_{n}^{2}\right\}\nonumber
		\\
		&\qquad\qquad\qquad\times\log\left(\frac{1}{\sqrt{2\pi s_{n}^{2}}}\cdot\mathrm{exp}\left\{-\left(k-\sum_{i=1}^{n}p_{i}\right)^{2}/2s_{n}^{2}\right\}
		\right)\nonumber\\
		&=\sum_{t=-[s_{n}^{2\alpha+1}]}^{[s_{n}^{2\alpha+1}]}\frac{1}{\sqrt{2\pi s_{n}^{2}}}\mathrm{e}^{-t^{2}/2s_{n}^{2}}\cdot\left[\frac{1}{2}\log\left(2\pi s_{n}^{2}\right)+\frac{t^{2}}{2s_{n}^{2}}\right]\nonumber\\
		&=\int_{-\infty}^{+\infty}\frac{\mathrm{e}^{-x^{2}}}{\sqrt{2\pi s_{n}^{2}}}dx\cdot\frac{1}{2}\log\left(2\pi s_{n}^{2}\right)\cdot\left(1+O\left(\frac{1}{s_{n}}\right)\right)\nonumber\\
		&\qquad\qquad\qquad+\frac{1}{2s_{n}^{2}}\cdot\int_{-\infty}^{+\infty}\frac{x^{2}\mathrm{e}^{-x^{2}}}{\sqrt{2\pi s_{n}^{2}}}dx\cdot\left(1+o\left(1\right)\right)\nonumber\\
		&	=\frac{1}{2}\log\left(2\pi\mathrm{e}s_{n}^{2}\right)+o\left(1\right).	
	\end{align}
	To conclude, we finish our proof by combining (4.6), (4.11)-(4.12) and Lemma 4.1.
\end{proof}

However, the pure Bernoulli distribution above is too special for $\left\{\tilde{\theta}_{n}\right\}_{n\in\mathbb{N}^{+}}$, as they are only supported on integer points, which cannot cover general lattice distributions as we want. For this, we introduce a sequence of independent lattice random variables $\left\{V_{n}\right\}_{n\in\mathbb{N}^{+}}$ on $\Xi$ with $l=1$ and a sequence of independent and identically Bernoulli distributed $\mathrm{Bern}\left(1/2\right)$ random variables $\left\{B_{n}\right\}_{n\in\mathbb{N}^{+}}$, which is independent of $\left\{V_{n}\right\}_{n\in\mathbb{N}^{+}}$, so that the $i$-th angle variable of the new sequence is expressed as $\tilde{\theta}_{i}=V_{i}+B_{i}$. Denote $\left\{\mu_{n}\right\}_{n\in\mathbb{N}^{+}}$ and $\left\{\eta_{n}^{2}\right\}_{n\in\mathbb{N}^{+}}$ as the expectations and finite variances of $\left\{V_{n}\right\}_{n\in\mathbb{N}^{+}}$, then the variance of $S_{n}$ can be currently calculated as $s_{n}^{2}=\sum_{i=1}^{n}\eta_{i}^{2}+n/4$. It can be found that under such case, we not only
maintain the symmetry of the Bernoulli term, but also allow the sequence $\left\{V_{n}\right\}_{n\in\mathbb{N}^{+}}$ to capture the asymmetric part of $\left\{\tilde{\theta}_{n}\right\}_{n\in\mathbb{N}^{+}}$. The Bernoulli randomness can be also understood as  ``smoothing'', just as the Gaussian smoothing step in continuous case. 

\begin{proposition}
	As $n \to \infty$, $H\left(S_{n}\right)-\log\left(s_{n}\right)\to\log\left(2\pi\mathrm{e}\right)/2$.
\end{proposition}
\begin{proof}
	Referring to \cite{MR815975}, it can be first calculated that
	\begin{align}
		D_{KL}\left(\frac{S_{n}+U}{\sqrt{s_{n}^{2}+\frac{1}{12}}}\right)
		&=D_{KL}\left(\frac{1}{\sqrt{2\left(s_{n}^{2}+\frac{1}{12}\right)}}\cdot\left(S_{n}+U\right)+\sqrt{\frac{1}{2}}\cdot Z\right)\nonumber\\
		&\qquad\qquad+\int^{\frac{1}{2}}_{0}J\left(\sqrt{\frac{1-t}{s_{n}^{2}+\frac{1}{12}}}\cdot\left(S_{n}+U\right)+\sqrt{t}\cdot Z\right)\frac{dt}{2\left(1-t\right)},
	\end{align}
	where the random variable $U$ is defined in Lemma 4.1, $J$ is the standardized Fisher information, see for example \cite{MR1122806,MR109101} for introductions, $Z\sim\mathcal{N}\left(\left(\sum_{i=1}^{n}\mu_{i}+n/2\right)/\sqrt{s_{n}^{2}+1/12},1\right)$ which is independent of $\left\{V_{n}\right\}_{n\in\mathbb{N}^{+}},\left\{B_{n}\right\}_{n\in\mathbb{N}^{+}}$ and $U$. If we at the same times define $W_{n}
	\sim\mathcal{N}\left(0,1/\left(12s_{n}^{2}+1\right)\right)$ and a sequence of independent random variables $\left\{Z_{n}\right\}_{n\in\mathbb{N}^{+}}$ with the $i$-th distributed as $\mathcal{N}\left(\mu_{i}+1/2,\eta_{i}^{2}+1/4\right)$, then we may choose to write $Z$ equivalently as $Z=\sum_{i=1}^{n}Z_{i}/\sqrt{2\left(s_{n}^{2}+1/12\right)}+W_{n}$.
	Thus, for $i=1,2,\cdots$, by defining $Y_{i}=\left(V_{i}+B_{i}+Z_{i}\right)/\sqrt{2}$
	with variance $\sigma_{Y_{i}}^{2}$, we obtain from the continuity of variable $U$ that 
	\begin{align}
		&D_{KL}\left(\frac{1}{\sqrt{2\left(s_{n}^{2}+\frac{1}{12}\right)}}\cdot\left(\sum_{i=1}^{n}\left(V_{i}+B_{i}\right)+U\right)+\sqrt{\frac{1}{2}}\cdot Z\right)\nonumber\\
		=&D_{KL}\left(\frac{\sum_{i=1}^{n}Y_{i}}{\sqrt{s_{n}^{2}+\frac{1}{12}}}+\frac{W_{n}}{\sqrt{2}}+\frac{U}{\sqrt{2\left(s_{n}^{2}+\frac{1}{12}\right)}}\right)\nonumber\\
		=&\frac{1}{2}\log\left(2\pi\mathrm{e}\right) -h\left(\sqrt{\frac{1}{2\left(s_{n}^{2}+\frac{1}{12}\right)}}\cdot\left(\sum_{i=1}^{n}\left(V_{i}+B_{i}\right)+U\right)+\frac{\sum_{i=1}^{n}Z_{i}}{\sqrt{2\left(s_{n}^{2}+\frac{1}{12}\right)}}+\frac{W_{n}}{\sqrt{2}}\right)\nonumber\\
		\le& \frac{1}{2}\log\left(2\pi\mathrm{e}\right)-h\left(\frac{\sum_{i=1}^{n}Y_{i}}{\sqrt{s_{n}^{2}+\frac{1}{12}}}\right)\nonumber\\
		=&D\left(\frac{\sum_{i=1}^{n}Y_{i}}{\sqrt{s_{n}^{2}+\frac{1}{12}}}\right)+\frac{1}{2}\log\left(2\pi\mathrm{e}\left(1+\frac{1}{12s_{n}^{2}+1}\right)\right)\nonumber\\
		\to&0,\ \ \ \ \mathrm{as}\ n \to \infty.
	\end{align}
	Meanwhile, for the second line of (4.13), 
	we rewrite the part inside $J$ as
	\begin{equation*}
		K_{n}:=\sqrt{\frac{1-t}{s_{n}^{2}+\frac{1}{12}}}\cdot\left(S_{n}+U\right)+\sqrt{t}\cdot\left(H_{n}+W_{n}\right),
	\end{equation*}
	with variance $\sigma_{K}^{2}$, where
	\begin{equation*}
		H_{n}\sim\mathcal{N}\left(\frac{\sum_{i=1}^{n}\mu_{i}+\frac{n}{2}}{\sqrt{s_{n}^{2}+\frac{1}{12}}},\frac{\sum_{i=1}^{n}\eta_{i}^{2}+\frac{n}{4}}{s_{n}^{2}+\frac{1}{12}}\right).
	\end{equation*}
	Thus, we have
	\begin{equation}
		0\le J\left(K_{n}\right)=\sigma^{2}_{K}I\left(K_{n}\right)-1\le \sigma^{2}_{K}I\left(\sqrt{\frac{1-t}{s_{n}^{2}+\frac{1}{12}}}\cdot S_{n}+\sqrt{t}H_{n}\right)-1,
	\end{equation}
	where $I$ is the corresponding Fisher information.
	In fact, through immediate calculations, $\sigma_{K}^{2}$ is exactly equal to 1. Moreover, since $\lim_{n \to \infty}s_{n}^{2}=\infty$, 
	for the sequence of independent random variables $\left\{\tilde{\theta}_{n}\right\}_{n\in\mathbb{N}^{+}}$, 
	we obtain that for any $\varepsilon>0$,
	\begin{align}
		&\frac{1}{s_{n}^{2}}\cdot\sum_{i=1}^{n}\mathbb{E}\left[\left(\tilde{\theta}_{i}-\left(\mu_{i}+\frac{1}{2}\right)\right)^{2}\cdot\mathbb{I}_{\left\{\left|\tilde{\theta}_{i}-\left(\mu_{i}+\frac{1}{2}\right)\right|\ge \varepsilon s_{n}\right\}}\right]\nonumber\\
		=&\frac{1}{s_{n}^{2}}\cdot\sum_{i=1}^{n}\mathbb{E}\left[\left(V_{i}-\mu_{i}+B_{i}-\frac{1}{2}\right)^{2}\cdot\mathbb{I}_{\left\{\left|\tilde{\theta}_{i}-\left(\mu_{i}+\frac{1}{2}\right)\right|\ge \varepsilon s_{n}\right\}}\right]\nonumber\\
		\le& \frac{1}{s_{n}^{2}}\cdot\sum_{i=1}^{n}\mathbb{E}\left[\left(V_{i}-\mu_{i}+1\right)^{2}\cdot\mathbb{I}_{\left\{\left|V_{i}-\mu_{i}\right|\ge \varepsilon s_{n}-1\right\}}\right]\nonumber\\
		\le& \frac{1}{s_{n}^{2}}\cdot\sum_{i=1}^{n}\left[\eta_{i}^{2}\cdot\mathbb{P}\left(\left|V_{i}-\mu_{i}\right|\ge\frac{ \varepsilon s_{n}}{2}\right)\right]+\frac{1}{s_{n}^{2}}\cdot\sum_{i=1}^{n}\mathbb{P}\left(\left|V_{i}-\mu_{i}\right|\ge\frac{ \varepsilon s_{n}}{2}\right)\nonumber\\
		&\qquad\qquad+\frac{2}{s_{n}^{2}}\cdot\sum_{i=1}^{n}\left[\mathbb{E}\left(V_{i}-\mu_{i}\right)^{2}\right]^{\frac{1}{2}}\cdot\left[\mathbb{P}\left(\left|V_{i}-\mu_{i}\right|\ge\frac{ \varepsilon s_{n}}{2}\right)\right]^{\frac{1}{2}}.
	\end{align}
	Since each $\mathbb{E}\left|V_{i}-\mu_{i}\right|^{2}=\eta_{i}^{2}<\infty$, we have $	\mathbb{P}\left(\left|V_{i}-\mu_{i}\right|\ge\frac{ \varepsilon s_{n}}{2}\right)\le4\eta_{i}^{2}/\varepsilon^{2} s_{n}^{2}$, 
	then for the third line of (4.16),
	\begin{equation*}
		\lim_{n\to \infty}\frac{1}{s_{n}^{2}}\cdot\sum_{i=1}^{n}\mathbb{P}\left(\left|V_{i}-\mu_{i}\right|\ge\frac{ \varepsilon s_{n}}{2}\right)\le  \lim_{n\to \infty}\frac{4\sum_{i=1}^{n}\eta^{2}_{i}}{\varepsilon^{2}s_{n}^{4}}=0.
	\end{equation*}
	Similarly, for the first and second lines of (4.16), as $n\to \infty$,
	\begin{align*}
		&\frac{1}{s_{n}^{2}}\cdot\sum_{i=1}^{n}\left[\eta_{i}^{2}\cdot\mathbb{P}\left(\left|V_{i}-\mu_{i}\right|\ge\frac{ \varepsilon s_{n}}{2}\right)\right]
		\le \frac{4\sum_{i=1}^{n}\eta_{i}^{4}}{\varepsilon^{2}s_{n}^{4}}\to 0,\\
		&\frac{2}{s_{n}^{2}}\cdot\sum_{i=1}^{n}\left[\mathbb{E}\left(V_{i}-\mu_{i}\right)^{2}\right]^{\frac{1}{2}}\cdot\left[\mathbb{P}\left(\left|V_{i}-\mu_{i}\right|\ge\frac{ \varepsilon s_{n} }{2}\right)\right]^{\frac{1}{2}}\le \frac{4\sum_{i=1}^{n}\eta_{i}^{2}}{\varepsilon^{2}s_{n}^{3}}\to0.
	\end{align*}
	Therefore, for any $\varepsilon>0$, we obtain
	\begin{equation*}
		\lim_{n \to \infty}\frac{1}{s_{n}^{2}}\cdot\sum_{i=1}^{n}\mathbb{E}\left[\left(\tilde{\theta}_{i}-\left(\mu_{i}+\frac{1}{2}\right)\right)^{2}\cdot\mathbb{I}_{\left\{\left|\tilde{\theta}_{i}-\left(\mu_{i}+\frac{1}{2}\right)\right|\ge \varepsilon s_{n}\right\}}\right]=0,
	\end{equation*}
	which leads to the Lindeberg condition for the sequence of independent random variables $\left\{\tilde{\theta}_{n}\right\}_{n\in\mathbb{N}^{+}}$. 
	Going further, the asymptotically Gaussian property of these two sequences is established.
	Back to (4.15), it can be estimated that 
	\begin{eqnarray*}
		0\le \lim_{n \to \infty}J\left(K_{n}\right)
		&\le \lim_{n \to \infty}\left[\left(1-t\right)\cdot \frac{s_{n}^{2}+\frac{1}{12}}{s_{n}^{2}}+ \frac{t\cdot\left(s_{n}^{2}+\frac{1}{12}\right)}{s_{n}^{2}}\right]-1=0,
	\end{eqnarray*}
	that is, $\lim_{n\to \infty}J\left(K_{n}\right)=0$, and, combining (4.14), we conclude that (4.13) converges to 0 as $n \to \infty$. 
	
	We are left to check the integrability of $J$. 
	If we define
	\begin{equation*}
		\hat{Z}\sim\mathcal{N}\left(\frac{\sum_{i=1}^{n}\mu_{i}}{\sqrt{s_{n}^{2}+\frac{1}{12}}},\frac{\sum_{i=1}^{n}\eta_{i}^{2}}{\sqrt{s_{n}^{2}+\frac{1}{12}}}\right), \ \ \ \  \tilde{Z}\sim\mathcal{N}\left(\frac{n}{2\sqrt{s_{n}^{2}+\frac{1}{12}}},\frac{\frac{1}{12}+\frac{n}{4}}{s_{n}^{2}+\frac{1}{12}}\right),
	\end{equation*} 
	and
	\begin{equation*}
		M_{n}=M_{n}\left(t\right):=\sqrt{\frac{1-t}{s_{n}^{2}+\frac{1}{12}}}\cdot\left(\sum_{i=1}^{n}B_{i}+U\right)+\sqrt{t}\tilde{Z},
	\end{equation*}
	with variance $\sigma_{M}^{2}$, then $K_{n}$ can be equivalently written as $K_{n}=M_{n}+\sqrt{\frac{1-t}{s_{n}^{2}+\frac{1}{12}}}\cdot\sum_{i=1}^{n}V_{i}+\sqrt{t}\hat{Z}$.
	Thus, it can be calculated that
	\begin{align}
		J\left(K_{n}\right)&=\sigma_{K}^{2}I\left(K_{n}\right)-1\nonumber\\
		&\le \left(\sigma_{M}^{2}+\frac{\sum_{i=1}^{n}\eta_{i}^{2}}{s_{n}^{2}+\frac{1}{12}}\right)I\left(M_{n}\right)-1\nonumber\\
		&=\frac{\sum_{i=1}^{n}\eta_{i}^{2}}{s_{n}^{2}+\frac{1}{12}}\cdot\left(J\left(M_{n}\right)+1\right)\cdot\frac{1}{\sigma^{2}_{M}}+J\left(M_{n}\right)\nonumber\\
		&=J\left(M_{n}\right)\cdot\left(1+\frac{\sum_{i=1}^{n}\eta_{i}^{2}}{s_{n}^{2}+\frac{1}{12}}\right)+\frac{\sum_{i=1}^{n}\eta_{i}^{2}}{s_{n}^{2}+\frac{1}{12}},
	\end{align}
	so that referring to \cite{MR815975},
	\begin{align}
		D_{KL}\left(\frac{\sum_{i=1}^{n}B_{i}+U}{\sqrt{s_{n}^{2}+\frac{1}{12}}}\right)
		&=D_{KL}\left(\frac{\sum_{i=1}^{n}B_{i}+U}{\sqrt{2\left(s_{n}^{2}+\frac{1}{12}\right)}}+\frac{\tilde{Z}}{\sqrt{2}}\right)\nonumber\\
		&\qquad\qquad+\int_{0}^{\frac{1}{2}}J\left(\sqrt{\frac{1-t}{s_{n}^{2}+\frac{1}{12}}}\cdot\left(\sum_{i=1}^{n}B_{i}+U\right)+\sqrt{t}\tilde{Z}\right)\frac{dt}{2\left(1-t\right)}\nonumber\\
		&=	D_{KL}\left(\frac{\sum_{i=1}^{n}B_{i}+U}{\sqrt{2\left(s_{n}^{2}+\frac{1}{12}\right)}}+\frac{\tilde{Z}}{\sqrt{2}}\right)+\int_{0}^{\frac{1}{2}}J\left(M_{n}\left(t\right)\right)\frac{dt}{2\left(1-t\right)}.
	\end{align}
	By taking advantage of Lemma 4.1 and Proposition 4.1, we estimate that 
	\begin{align*}
		&D_{KL}\left(\frac{\sum_{i=1}^{n}B_{i}+U}{\sqrt{s_{n}^{2}+\frac{1}{12}}}\right),\ \ D_{KL}\left(\frac{\sum_{i=1}^{n}B_{i}}{\sqrt{s_{n}^{2}+\frac{1}{12}}}\right),\\ &D_{KL}\left(\frac{\sum_{i=1}^{n}B_{i}+U}{\sqrt{s_{n}^{2}+\frac{1}{12}}}+\tilde{Z}\right),\ \ D_{KL}\left(\frac{\sum_{i=1}^{n}B_{i}}{\sqrt{s_{n}^{2}+\frac{1}{12}}}+\tilde{Z}\right)
	\end{align*}
	all converge to 0 as $n \to \infty$.
	Thus, we conclude from (4.17)-(4.18) that $J\left(K_{n}\right)$ together with $J\left(M_{n}\right)$ is uniformly integrable on $(0,1/2)$.
	Finally, by adopting the definition of $U$ and Theorem 4.1 again, we deduce that as $n \to \infty$,
	\begin{eqnarray*}
		D_{KL}\left(\frac{S_{n}+U}{s_{n}}\right)=D_{KL}\left(\frac{S_{n}}{s_{n}}\right)
		=\frac{1}{2}\log\left(2\pi\mathrm{e}\right)-\left[H\left(S_{n}\right)-\log \left(s_{n}\right)\right]
		\to0.
	\end{eqnarray*}
	Our proof is complete now.
\end{proof}

\begin{remark}
	In fact, the Bernoulli sequence $\left\{B_{n}\right\}_{n\in\mathbb{N}^{+}}$ can be also fulfilled with different probability parameters just like in Proposition 4.2, which will not actually affect the result of Proposition 4.3, as long as 
	the condition $\lim_{n \to \infty}\sum_{i=1}^{n}\left(\eta_{i}^{2}+p_{i}\left(1-p_{i}\right)\right)=\infty$ is set. Since
	deriving our final conclusion requires focusing on the sequence
	involving independent and identically distributed   $\left\{B_{n}\right\}_{n\in\mathbb{N}^{+}}$, we will proceed with our analysis of such kind of sequence, aided by Proposition 4.1 and Proposition 4.3.
\end{remark}

It is worth noting that not all lattice-valued random variables admit an explicit decomposition of the form $\tilde{\theta}_{n}=V_{n}+B_{n}$. To address this limitation, 
for any lattice variable defined on $\Xi$ with  $l=1$, we write:
\begin{equation}
	\tilde{\theta}_{n}\sim V_{n}+W_{n}B_{n},\ \ \ \ n=1,2,\cdots,
\end{equation}
which is fulfilled with probability mass function $p_{n}$, where $V_{n}\in \mathbb{Z}$, $B_{n}\sim\mathrm{Bern}\left(1/2\right)$ and $W_{n}\sim\mathrm{Bern}\left(q_{n}\right)$, and they are independent of each other. Here, (4.19) is explicitly decomposed into the lattice part $\left\{V_{n}\right\}_{n \in\mathbb{N}^{+}}$ and independent Bernoulli noise $\left\{B_{n}\right\}_{n\in\mathbb{N}^{+}}$, the sequence $\left\{W_{n}\right\}_{n\in\mathbb{N}^{+}}$ plays the role of determining whether a Bernoulli component exists in the distribution of the independent (not necessarily identically distributed) sequence $\left\{\tilde{\theta}_{n}\right\}_{n\in\mathbb{N}^{+}}$.
If we denote $p_{V_{i},W_{i}}$ as the joint probability mass function of $V_{i}$ and $W_{i}$, then it can be deduced that for $k\in \mathbb{Z}$ and $i=1,2,\cdots,n$,
\begin{align*}
	\mathbb{P}\left(X_{i}=k\right)&=p_{i}\left(k\right)\\
	&=p_{V_{i},W_{i}}\left(k,0\right)+\frac{p_{V_{i},W_{i}}\left(k,1\right)}{2}+\frac{p_{V_{i},W_{i}}\left(k-1,1\right)}{2}\\
	&=p_{i}\left(k\right)-\frac{1}{2}\left(p_{V_{i},W_{i}}\left(k,1\right)+p_{V_{i},W_{i}}\left(k-1,1\right)\right)\nonumber\\
	&\qquad\qquad\qquad\qquad+\frac{p_{V_{i},W_{i}}\left(k,1\right)}{2}+\frac{p_{V_{i},W_{i}}\left(k-1,1\right)}{2},
\end{align*}
with $p_{V_{i},W_{i}}\left(k,1\right)/2\le p_{i}\left(k\right) $ and $p_{V_{i},W_{i}}\left(k-1,1\right)/2\le p_{i}\left(k+1\right)$.
Therefore, we obtain that for $i=1,2,\cdots,n$, 
\begin{align}
	&p_{V_{i},W_{i}}\left(k,1\right)=\min\left\{p_{i}\left(k\right),p_{i}\left(k+1\right)\right\}\nonumber,\\
	&p_{V_{i},W_{i}}\left(k,0\right)=p_{i}\left(k\right)-\frac{1}{2}\left[p_{V_{i},W_{i}}\left(k,1\right)+p_{V_{i},W_{i}}\left(k-1,1\right)\right],\nonumber\\
	&q_{i}=\sum_{k}\min\left\{p_{V_{i},W_{i}}\left(k-1,1\right),p_{V_{i},W_{i}}\left(k,1\right)\right\}>0,	
\end{align}
where the inequality in (4.20) comes from the fact that  $l=1$. 

\begin{lemma}
	There is a random variable $V^{(n)}$ with values in $\left\{na+kl:k\in\mathbb{Z} \right\}$ and a $W^{(n)} \sim\mathrm{Bern}\left(q^{(n)}\right)$,
	such that $S_{n}\sim V^{(n)}+W^{(n)}B$, where $B\sim\mathrm{Bern}\left(1/2\right)$ is independent of $\left(V^{(n)},W^{(n)}\right)$ and $q^{(n)}\to 1$ as $n \to \infty$. Furthermore, $\mathrm{Var}\left(S_{n}\mid W^{(n)}=1\right)=\sum_{i=1}^{n}\mathrm{Var}\left(\tilde{\theta}_{i}\right)\left(1+o\left(1\right)\right)$.
\end{lemma}
\begin{proof}
	The proof can be found in Appendix.
\end{proof}

Now we are ready to provide our  main result and proof. 
\begin{theorem}
	As $n \to \infty$, the sequence $\left\{\tilde{\theta}_{n} \right\}_{n\in\mathbb{N}^{+}}$ satisfies that
	\begin{align*}
		\lim_{n \to \infty}D_{KL}\left(S_{n}\right)=\lim_{n \to \infty}\left[H\left(S_{n}\right)+\sum_{j=1}^{d}\log l_{j}-\frac{1}{2}\log\left(\left(2\pi\mathrm{e}\right)^{d}\cdot\mathrm{det}\left(\Sigma_{n}^{2}\right)\right)\right]=0,
	\end{align*}
	where $\Sigma^{2}_{n}$ is a $d\times d$ diagonal matrix defined in (4.27). Meanwhile, there exists some $d$-dimensional quantized Gaussian random variable $Z$ matching the variance and expectation of $S_{n}$, such that
	\begin{equation*}
		\lim_{n \to \infty}\left\|S_{n}-Z\right\|_{TV}=0.
	\end{equation*}
\end{theorem}
\begin{proof}
	Suppose that $l=1$ and adopt the notation 
	$s_{n}^{2}=\sqrt{\sum_{k=1}^{n}\mathrm{Var}\left(\tilde{\theta}_{k}\right)}$. For $d=1$, 
	we choose a large enough integer $M$ such that when $1\le i \le [n/M]$, 
	\begin{equation*}
		S_{n}=\sum_{i=1}^{[n/M]}\sum_{j=\left(i-1\right)M+1}^{iM}\tilde{\theta}_{j}=\sum_{i=1}^{[n/M]}S_{i}^{(M)}:=\sum_{i=1}^{[n/M]}\left(V_{i}^{(M)}+W_{i}^{(M)}B_{i}\right).
	\end{equation*}
	Specifically, for $i=1,2,\cdots,[n/M]$, $W_{i}^{(M)}=\mathbb{I}_{\left\{\sum_{j=\left(i-1\right)M+1}^{iM}W_{j}\ge 1\right\}}\sim\mathrm{Bern}\left(q_{i}^{(M)}\right)$, 
	where $q_{i}^{(M)}=1-\prod_{j=\left(i-1\right)M+1}^{iM}\left(1-q_{j}\right)$.
	Furthermore, we introduce $W^{(M)}=\left(W^{(M)}_{1},\cdots,W^{(M)}_{[n/M]}\right)$ and set $A_{M}$ as 
	\begin{equation*}
		A_{M}:=\left\{w\mid w=\left(w_{1},\cdots,w_{[n/M]}\right)\in \left\{0,1\right\}^{[n/M]} \right\},
	\end{equation*}
	with the condition that there are at least $ \sum_{i=1}^{[n/M]}q_{i}^{(M)}-[n\epsilon/2M]$ values of $i$ satisfying $w_{i}=1$, where $\epsilon>0$ is arbitrary.
	Thus, if we at the same time denote $\bar{W}^{(M)}$ as a vector such that its $i$-th components $\bar{W}_{i}^{(M)}$ are all equal to 1, which is a special case of $A_{M}$, then the lower bound of $H\left(S_{n}\right)$ can be estimated to satisfy
	\begin{align}
		H\left(S_{n}\right)
		&\ge H\left(\sum_{i=1}^{[n/M]}\left(V_{i}^{(M)}+W^{(M)}_{i}B_{i}\right)\mid W_{1}^{(M)},\cdots,W^{(M)}_{[n/M]}\right)\nonumber\\
		&\ge	 \sum_{w\in A_{M}}\mathbb{P}\left(W^{(M)}=w\right)\cdot H\left(\sum_{1\le i \le\left[\frac{n}{M}\right]:\ w_{i}=1 }\left(V_{i}^{(M)}+W^{(M)}_{i}B_{i}\right)\mid W^{(M)}=w\right)\nonumber\\
		&\ge	 \mathbb{P}\left(\sum_{i=1}^{[n/M]}W_{i}^{(M)}\ge\sum_{i=1}^{[n/M]}q_{i}^{(M)}-\left[\frac{n\epsilon}{2M}\right] \right)\nonumber5\\
		&\qquad\qquad\times H\left(\sum_{l=1}^{\sum_{j=1}^{[n/M]}q_{j}^{(M)}-\left[\frac{n\epsilon}{2M}\right]}\left(V_{i}^{(M)}+B_{i}\right)\mid \bar{W}^{(M)}=\mathbf{1}\right)\nonumber\\
		&\ge \left(1-\mathrm{exp}\left\{-2\left[\frac{n}{M}\right]^{-1}\left[\frac{n\epsilon}{2M}\right]^{2}\right\}\right)\nonumber\\
		&\qquad\qquad\times\frac{1}{2}\log\left(2\pi\mathrm{e}\sum_{i=1}^{\sum_{j=1}^{[n/M]}q_{j}^{(M)}-\left[\frac{n\epsilon}{2M}\right]}\mathrm{Var}\left(V_{i}^{(M)}+B_{i}\mid W_{i}^{(M)}=1\right)\right) \\
		&\ge \frac{1}{2}\log\left(2\pi\mathrm{e}\left(\sum_{i=1}^{\sum_{j=1}^{[n/M]}q_{j}^{(M)}-\left[\frac{n\epsilon}{2M}\right]}\mathrm{Var}\left(V_{i}^{(M)}+B_{i}\mid W_{i}^{(M)}=1\right)\right)\right) -o\left(1\right)\nonumber\\
		&= \frac{1}{2}\log\left(2\pi\mathrm{e}\left(\sum_{i=1}^{\sum_{j=1}^{[n/M]}q_{l,j}^{(M)}-\left[\frac{n\epsilon}{2 M}\right]}\sum_{k=\left(i-1\right)M+1}^{iM}\mathrm{Var}\left(\tilde{\theta}_{k}\right)\left(1-\epsilon\right)\right)\right) -o\left(1\right)\\
		&=\frac{1}{2}\log\left(2\pi\mathrm{e}\left(\sum_{i=1}^{M\sum_{j=1}^{[n/M]}q_{j}^{(M)}-M\left[\frac{n\epsilon}{2M}\right]}\mathrm{Var}\left(\tilde{\theta}_{k}\right)\right)\right)+\frac{1}{2}\log\left(1-\epsilon\right) -o\left(1\right),
	\end{align}
	where (4.21) is obtained from the Hoeffding inequality \cite{MR144363} and Theorem 4.2,  while (4.22) is deduced from Proposition 4.1. More precisely, if we choose $M$ to be large enough such that $	\sum_{j=1}^{[n/M]}q_{j}^{(M)}\ge n\left(1-\epsilon/2\right)/M$,
	this choice leads to
	\begin{equation*}
		M\sum_{j=1}^{[n/M]}q_{j}^{(M)}-M\left[\frac{n\epsilon}{2M}\right]\ge n\left(1-\frac{\epsilon}{2}\right)-\frac{n\epsilon}{2}\ge \left[n\left(1-\epsilon\right)\right].
	\end{equation*}
	Meanwhile, under the boundedness of each $\mathrm{Var}\left(\tilde{\theta}_{i}\right)\ge 1/4$, we obtain $\sum_{i=\left[n\left(1-\epsilon\right)\right]}^{n}\mathrm{Var}\left(\tilde{\theta}_{i}\right)/s_{n}^{2}=O\left(\epsilon\right)$,
	then it can be deduced that
	\begin{equation*}
		\sum_{i=1}^{\left[n\left(1-\epsilon\right)\right]}\mathrm{Var}\left(\tilde{\theta}_{i}\right)=s_{n}^{2}-\sum_{i=\left[n\left(1-\epsilon\right)\right]+1}^{n}\mathrm{Var}\left(\tilde{\theta}_{i}\right)=\left(1-O\left(\epsilon\right)\right)\cdot s_{n}^{2}.
	\end{equation*}
	Based on these above, (4.23) turns to
	\begin{align}
		H\left(S_{n}\right)&\ge \frac{1}{2}\log\left(2\pi\mathrm{e}\left(\sum_{i=1}^{\left[n\left(1-\epsilon\right)\right]}\mathrm{Var}\left(\tilde{\theta}_{i}\right)\right)\right)+\frac{1}{2}\log\left(1-\epsilon\right) -o\left(1\right)\nonumber\\
		&= \frac{1}{2}\log\left(2\pi\mathrm{e}s_{n}^{2}\right)+\log\left(1-O\left(\epsilon\right)\right) -o\left(1\right).
	\end{align}
	At the same time, in virtue of Proposition 4.3, we can immediately deduce that 
	\begin{equation}
		H\left(S_{n}\right)\le \frac{1}{2}\log\left(2\pi\mathrm{e}s_{n}^{2}\right)-\log l+o\left(1\right).
	\end{equation}
	Combining (4.24)-(4.25), we obtain 
	\begin{equation}
		\lim_{n \to \infty}\left[H\left(S_{n}\right)+\log l-\frac{1}{2}\left(2\pi\mathrm{e}s_{n}^{2}\right)\right]=0,\ \ \ \ \lim_{n \to \infty}D\left(S_{n}\right)=0,
	\end{equation}
	from Lemma 4.1 and the arbitrariness of $\epsilon>0$.
	Notably, Pinsker's inequality directly implies convergence in total variation norm whenever the relative entropy tends to zero. Therefore, 
	based on (4.26), we are naturally led to
	\begin{align*}
		\left\|S_{n}-Z\right\|_{TV}&\le \left\|S_{n}-Z_{n}\right\|_{TV}+\left\|Z_{n}-Z\right\|_{TV}
		\le\sqrt{
			\frac{1}{2}D\left(S_{n}\right)}+ \left\|Z_{n}-Z\right\|_{TV}
		\to 0,\ \ \mathrm{as}\ n \to \infty,	
	\end{align*}
	where $Z_{n}$ is constructed as the quantised Gaussian for the definition of $D_{KL}\left(S_{n}\right)$ and $Z$ follows a Gaussian distribution with the same expectation and finite variance as $S_{n}$.
	
	Next, we extend both the process $\left\{\tilde{\theta}_{n}\right\}_{n\in\mathbb{N}^{+}}$ as well as its partial sum to arbitrary finite dimensions $d \ge 1$. To build this, we suppose that the $i$-th asymptotic covariance matrix is denoted as $s^{2}_{n,i}<\infty$ 
	and construct $\Xi^{(d)}:=\Xi_{1}\times\cdots\times\Xi_{d}$.
	Thus, the vector $\textbf{l}=\left(l_{1},\cdots,l_{d}\right)$  is naturally introduced for the $d$-dimensional lattice. 
	Then  
	it can be estimated that
	\begin{align*}
		H\left(S_{n}\right)&=-\sum_{\textbf{s}\in \Xi^{(d)}}\left(\prod_{i=1}^{d}p_{i}\left(s^{(i)}\right)\right)\log\left(\prod_{i=1}^{d}p_{i}\left(s^{(i)}\right)\right)\\
		&=-\sum_{\textbf{s}\in \Xi^{(d)}}\left(\prod_{i=1}^{d}p_{i}\left(s^{(i)}\right)\right)\cdot\sum_{j=1}^{d}\log\left(p_{j}\left(s^{(j)}\right)\right)\\
		&=-\sum_{j=1}^{d}\sum_{\textbf{s}\in \Xi^{(d)}}\left(\prod_{i=1}^{d}p_{i}\left(s^{(i)}\right)\right)\log\left(p_{j}\left(s^{(j)}\right)\right)\\
		&=-\sum_{j=1}^{d}\left(\sum_{s^{(j)}\in \Xi_{j}}p_{j}\left(s^{(j)}\right)\log\left(p_{j}\left(s^{(j)}\right)\right)\right)\cdot\prod_{i\ne j}\sum_{s^{(i)}}p_{i}\left(s^{(i)}\right)\\
		&=\sum_{j=1}^{d}H\left(S_{n,j}\right),	
	\end{align*}
	where $\textbf{s}=\left(s^{(1)},\cdots,s^{(d)}\right)$, $S_{n}=\left(S_{n,1},\cdots,S_{n,d}\right)$ and $p_{i}$ is the probability mass function of $S_{n,i}$. 
	It leads to
	\begin{align*}
		&\lim_{n \to \infty}\left[H\left(S_{n}\right)+\sum_{j=1}^{d}\log l_{j}-\frac{1}{2}\log\left(\left(2\pi\mathrm{e}\right)^{d}\cdot\mathrm{det}\left(\Sigma_{n}^{2}\right)\right)\right]\nonumber\\
		=&\lim_{n \to \infty}\left[\sum_{j=1}^{d}H\left(S_{n,j}\right)+\sum_{j=1}^{d}\log l_{j}-\frac{1}{2}\log\left(\left(2\pi\mathrm{e}\right)^{d}\cdot\mathrm{det}\left(\Sigma_{n}^{2}\right)\right)\right]\nonumber\\
		=&\lim_{n \to \infty}\sum_{j=1}^{d}\left(\frac{1}{2}\log\left(2\pi\mathrm{e}s_{n,j}^{2}\right)-\log
		l_{j}\right)+\sum_{j=1}^{d}\log l_{j}-\frac{1}{2}\log\left(\left(2\pi\mathrm{e}\right)^{d}\cdot\mathrm{det}\left(\Sigma_{n}^{2}\right)\right)\nonumber\\
		=&0,\nonumber	
	\end{align*}
	where 
	\begin{equation}
		\Sigma_{n}^{2}=\mathrm{diag}\left(s_{n,1}^{2},\cdots,s_{n,d}^{2}\right)
	\end{equation}
	is a $d\times d$ diagonal matrix for $S_{n}$. It can be similarly deduced from (4.26) and the independence between components of each $\tilde{\theta}_{n}$ that $\lim_{n \to 0}D_{KL}\left(S_{n}\right)=	\lim_{n \to 0}\sum_{j=1}^{d}D_{KL}\left(S_{n,j}\right)=0$.
	We finish our proof.
\end{proof}

\begin{remark}
The lattice spacing $l_{i}$
	can vary across dimensions, explicitly capturing scale differences along different directions. Moreover, the sum $\sum_{j=1}^{d}\log l_{j}$ appearing in the Shannon entropy difference is a technical correction. 
	Its role is to account for the discretization, specifically, to adjust for the entropy contribution inherent to the spacing between lattice points, which arises due to the quantization of the fluctuation sequence. For uniform spacing $l_{j}=1$
	across all dimensions, this term vanishes entirely. 
	Importantly, this correction does not alter the core physical conclusion, it merely ensures asymptotic convergence holds under discretization, refining the comparison with the continuous Gaussian approximation without obscuring the underlying dynamics.
\end{remark}

\begin{remark}
	Since  $\log\left(\left(2\pi\mathrm{e}\right)^{d}\cdot\mathrm{det}\left(\Sigma^{2}_{n}\right)\right)/2$
	denotes the Shannon entropy of a $d$-dimensional Gaussian distribution with covariance matrix $\Sigma_{n}^{2}$, Theorem 4.1 thus reveals that the macroscopic accumulation $S_{n}$
	drives the system toward a state of maximal randomness attainable for a given variance. The convergence of Shannon entropy to that of a Gaussian quantity signifies that, in the space of this observable, the system's uncertainty reaches the theoretical maximum for any distribution with the same second moment. Furthermore, the vanishing relative entropy demonstrates a stronger fact: the distribution of $S_{n}$
	becomes information-theoretically indistinguishable from that of its Gaussian counterpart with matched variance. All higher-order statistical structures (temporal correlations, non-Gaussian features) present in the microscopic fluctuation path $\left\{\tilde{\theta}_{n}\right\}_{n\in \mathbb{N}^{+}}$ are asymptotically erased by the summing operation, leaving only the variance as the sole determinant of the macroscopic information content.
\end{remark}

For notational clarity, in what follows, we reintroduce the superscript ``$\left(l\right)$'' or subscript ``$l$'' to denote the discretization scale
for related quantities.  

\begin{corollary}
	The sequence $\left\{\hat{\theta}_{n}\right\}_{n\in \mathbb{N}^{+}}$  satisfies that
	\begin{equation*}
		\lim_{n \to \infty}\left[h\left(\sum_{i=1}^{n}\hat{\theta}_{i}\right)-\frac{1}{2}\log\left(\mathrm{det}\left(\hat{\Sigma}_{n}^{2}\right)\right)\right]= \frac{d}{2}\cdot\log\left(2\pi\mathrm{e}\right),
	\end{equation*}
	where the covariance matrix $\hat{\Sigma}_{n}^{2}$ is defined in (4.29). Moreover, $\lim_{n \to \infty}D_{KL}\left(\sum_{i=1}^{n}\hat{\theta}_{i}\right)=0$,
	while there exists some $d$-dimensional Gaussian random variable $\Upsilon\sim\mathcal{N}\left(0,\hat{\Sigma}^{2}_{n}\right)$ written as $\Upsilon=\left(\Upsilon_{1},\cdots,\Upsilon_{d}\right)$ such that $\lim_{n \to \infty}\left\|\sum_{i=1}^{n}\hat{\theta}_{i}-\Upsilon\right\|_{TV}=0$.
\end{corollary}
\begin{proof}
	We start with the one-dimensional case. Specifically, by taking advantage of Theorem 4.1 and the relationship between differential entropy and discrete entropy \cite{MR1122806}, we obtain that  
	\begin{align}
		\lim_{n \to \infty}\left[h\left(\sum_{k=1}^{n}\hat{\theta}_{k}\right)-\frac{1}{2}\log\left(2\pi\mathrm{e}\bar{s}_{n}^{2}\right)\right]&=\lim_{n \to \infty}\lim_{l \to 0}\left[H\left(S^{(l)}_{n}\right)+\log l-\frac{1}{2}\log\left(2\pi\mathrm{e}\bar{s}_{n}^{2}\right)\right]\nonumber\\
		&=\lim_{l \to 0}\lim_{n \to \infty}\left[H\left(S^{(l)}_{n}\right)+\log l-\frac{1}{2}\log\left(2\pi\mathrm{e}\bar{s}_{n}^{2}\right)\right]\nonumber\\
		&=\lim_{l \to 0}\lim_{n\to \infty}\left[\frac{1}{2}\log\left(2\pi\mathrm{e}\left(s_{n}^{(l)}\right)^{2}\right)-\frac{1}{2}\log\left(2\pi\mathrm{e}\bar{s}_{n}^{2}\right)\right]\nonumber\\
		&=0,	
	\end{align}
	where $	\bar{s}_{n}^{2}:=\lim_{l \to 0}\left(s_{n}^{(l)}\right)^{2}$.
	We are required to state the rationality of $\bar{s}_{n}^{2}$. Specifically,
	\begin{align*}
		\mathrm{Var}\left(S^{(l)}_{n}\right)-\mathrm{Var}\left(\sum_{k=1}^{n}\hat{\theta}_{k}\right)&=\mathrm{Var}\left(\sum_{k=1}^{n}\varsigma_{k}^{(l)}\right)+2\mathrm{Cov}\left(\sum_{k=1}^{n}\hat{\theta}_{k},\sum_{k=1}^{n}\varsigma_{k}^{(l)}\right)\\
		&=\sum_{k=1}^{n}\mathrm{Var}\left(\varsigma_{k}^{(l)}\right)+2\sum_{k=1}^{n}\sum_{i=1}^{n}\mathrm{Cov}\left(\hat{\theta}_{k},\varsigma_{i}^{(l)}\right)\\&\le \frac{nl^{2}}{4}+\sum_{k=1}^{n}\sum_{i=1}^{n}\sqrt{\mathrm{Var}\left(\hat{\theta}_{k}\right)\mathrm{Var}\left(\varsigma_{i}^{(l)}\right)}\\
		&\le \frac{nl^{2}}{4}+\pi l\\
		&\to 0,\ \ \ \ \mathrm{as}\ \ l \to 0,
	\end{align*}
	where $\sum_{k=1}^{n}\varsigma^{(l)}_{k}:= S_{n}^{(l)}-\sum_{k=1}^{n}\hat{\theta}_{k}$.  Therefore, we deduce from (4.28) that $D_{KL}\left(\sum_{k=1}^{n}\hat{\theta}_{k}\right)$ converges to 0, where the corresponding Gaussian distribution in the relative entropy has the same variance $\bar{s}_{n}^{2}$ and expectation with $\sum_{k=1}^{n}\hat{\theta}_{k}$. 
	Moreover, for some Gaussian random variable $\Upsilon\sim\mathcal{N}\left(0,\bar{s}_{n}^{2}\right)$, we have $\lim_{n \to \infty}\left\|\sum_{k=1}^{n}\hat{\theta}_{k}-\Upsilon \right\|_{\mathrm{TV}}=0$  immediately.  Next, we raise the dimension of the system to any finite integer $d \ge 1$. Analogous to the one-dimensional case as above, the $j$-th component $s_{l,n,j}^{2}$ (representing the variance of $\tilde{\theta}_{n,j}^{(l)}$)
	considered here admits a limiting value $\bar{s}_{n,j}^{2}$ as $l\to 0$. The resulting $d\times d$ diagonal limit  matrix
	\begin{equation}
		\hat{\Sigma}_{n}^{2}:=\mathrm{diag}\left(\bar{s}_{n,1}^{2},\cdots,\bar{s}_{n,d}^{2}\right)
	\end{equation} 
	is	thus reasonable to govern the asymptotic behavior of the process $\left\{\hat{\theta}_{n}\right\}_{n\in\mathbb{N}^{+}}$ from the point of view of entropy, which leads us to obtain that
	\begin{align*}
		&\lim_{n \to \infty}\left[h\left(\sum_{k=1}^{n}\hat{\theta}_{k}\right)-\frac{1}{2}\log\left(\left(2\pi\mathrm{e}\right)^{d}\cdot\mathrm{det}\left(\hat{\Sigma}_{n}^{2}\right)\right)\right]\\
		=&\lim_{\textbf{l}\to 0}\lim_{n \to \infty}\left[H\left(S_{n}^{(l)}\right)+\sum_{j=1}^{d}\log l_{j}-\frac{1}{2}\log\left(\left(2\pi\mathrm{e}\right)^{d}\cdot\mathrm{det}\left(\hat{\Sigma}_{n}^{2}\right)\right)\right]\\
		=&\lim_{\textbf{l}\to 0}\lim_{n \to \infty}\left[\frac{1}{2}\log\left(\left(2\pi\mathrm{e}\right)^{d}\cdot\mathrm{det}\left(\Sigma_{l,n}^{2}\right)\right)-\frac{1}{2}\log\left(\left(2\pi\mathrm{e}\right)^{d}\cdot\mathrm{det}\left(\hat{\Sigma}_{n}^{2}\right)\right)\right]\\
		=&0,
	\end{align*}
	together with $	\lim_{n \to \infty}D_{KL}\left(\sum_{k=1}^{n}\hat{\theta}_{k}\right)=0$.
	Our proof is complete now.
\end{proof}

\section{Conclusion}

In this work, we have studied fluctuation processes generated by
stochastically perturbed integrable Hamiltonian systems from both
dynamical and probabilistic 
viewpoints. We have shown that
the properly rescaled accumulated observables admit Gaussian
fluctuation descriptions, providing a rigorous characterization of the
asymptotic statistical behavior induced by stochastic perturbations.
Beyond the distributional properties of individual observables, we have
introduced a lattice-based symbolic representation of the fluctuation
process and investigated the entropy structure of the associated
constrained path space. By analyzing the corresponding shift dynamics,
we established an entropy variational characterization and showed that
the Gaussian product measure realizes the maximal entropy under the
prescribed asymptotic constraint.

Several directions remain open for future investigation. One natural
problem is to extend the present framework to more general classes of
non-integrable or weakly chaotic Hamiltonian systems, where the
interaction between deterministic dynamics and stochastic perturbations
may produce richer fluctuation structures. Another direction is to
develop entropy characterizations under more general constraints,
including non-Gaussian fluctuation regimes and alternative notions of
dynamical complexity.

\section{Appendix}
\begin{proof}[Proof of Lemma 4.2]
	It can be first deduced from (4.20) that
	\begin{align*}
		S_{n}&\sim\sum_{i=1}^{n}V_{i}+\sum_{i=1}^{N_{n}}B_{i}
			=\left[\sum_{i=1}^{n}V_{i}+\left(\sum_{i=1}^{N_{n}-1}B_{i}\right)\mathbb{I}_{\left\{N_{n}\ge 1\right\}}\right]+\mathbb{I}_{\left\{N_{n}\ge 1\right\}}B
		:= V^{(n)}+W^{(n)},	
	\end{align*}
	where $N_{n}=\sum_{i=1}^{n}W_{i}$, $B\sim\mathrm{Bern}\left(1/2\right)$, and, $q^{(n)}=1-\prod_{i=1}^{n}(1-q_{i})\to 1$ as $n \to \infty$.
	Therefore, 
	\begin{align*}
		\mathbb{E}\left(	S_{n}\right)&=\mathbb{E}\left[\sum_{i=1}^{n}V_{i}\mid W^{(n)}=0\right]\prod_{i=1}^{n}\left(1-q_{i}\right)+\mathbb{E}\left[\sum_{i=1}^{n}X_{i}\mid W^{(n)}=1\right]\left(1-\prod_{i=1}^{n}\left(1-q_{i}\right)\right)=0,	
	\end{align*}
	while
	\begin{align*}
		\mathbb{E}	S_{n}^{2}
		&=\mathbb{E}\left[\left(\sum_{i=1}^{n}V_{i}\right)^{2}\mid W^{(n)}=0\right]\prod_{i=1}^{n}\left(1-q_{i}\right)
	+\mathbb{E}\left[\left(\sum_{i=1}^{n}X_{i}\right)^{2}\mid W^{(n)}=1\right]\left(1-\prod_{i=1}^{n}\left(1-q_{i}\right)\right)\\
		&=\sum_{i=1}^{n}\mathrm{Var}\left(\tilde{\theta}_{i}\right).	
	\end{align*}
	Under the fact that
	\begin{equation*}
		\mathbb{E}\left[\sum_{i=1}^{n}V_{i}\mid W^{(n)}=0\right]=O\left(n\right),\ \ \ 
		\mathbb{E}\left[\left(\sum_{i=1}^{n}V_{i}\right)^{2}\mid W^{(n)}=0\right]=O\left(n^{2}\right),
	\end{equation*}
	we obtain from the convergence of $q^{(n)}$ that
	\begin{equation*}
		\mathbb{E}\left[	S_{n}\mid W^{(n)}=1\right]=o\left(1\right) ,\ \ \ \	\mathbb{E}\left[	S_{n}^{2}\mid W^{(n)}=1\right]=\sum_{i=1}^{n}\mathrm{Var}\left(\tilde{\theta}_{i}\right)\left(1+o\left(1\right)\right).
	\end{equation*}
	Our proof is complete.
\end{proof}

\bibliographystyle{plain}

\end{document}